%% file: samplepaper.tex
\documentclass[orivec,runningheads]{llncs}

\usepackage[T1]{fontenc}
\usepackage{graphicx}

\usepackage{amsmath}
\usepackage{amssymb}
\usepackage{bm}

\usepackage[authoryear,round]{natbib}

\usepackage{hyperref}
\usepackage[nameinlink,noabbrev,capitalise]{cleveref}

\usepackage{algorithm}
\usepackage{algpseudocode}

\input{article_preamble.tex}

\begin{document}
	
	\title{Exploiting Variable Implications in Presolve for Mixed Integer Programming}
	
	\titlerunning{Exploiting Variable Implications in Presolve for Mixed Integer Programming}
	
	\author{
		Wei-Kun Chen\inst{1}
		\and
		Chang-Long Li\inst{1}
		\and
		Zhao-Wei Wang\inst{2}
		\and
		Yu-Hong Dai\inst{2}
		\and
		Zi-Shuo Li\inst{3}
		\and
		Meng Lu\inst{3}
	}
	
	\authorrunning{Chen et al.}
	
	\institute{
		School of Mathematics and Statistics,
		Beijing Institute of Technology,
		100081 Beijing, China\\
		\email{\{chenweikun, lichanglong\}@bit.edu.cn}
		\and
		Academy of Mathematics and Systems Science,
		Chinese Academy of Sciences,
		100190 Beijing, China\\
		School of Mathematical Sciences,
		University of Chinese Academy of Sciences,
		100049 Beijing, China\\
		\email{\{wangzhaowei, dyh\}@lsec.cc.ac.cn}
		\and
		Taylor Lab,
		Huawei Technologies Co., Ltd.,
		China\\
		\email{\{lizishuo4, lumeng22\}@huawei.com}
	}
	
	\maketitle
	
	\begin{abstract}
		
		Presolve for mixed integer programming (MIP) problems aims to eliminate redundant information, strengthen the formulation, and extract useful structural information for the subsequent branch-and-cut process.
		An important type of such structural information is the variable implications (VIs), which describe how a bound on a variable depends on a bound of a binary variable.	
		In this paper, we develop two new presolve techniques that exploit VIs to derive reductions for MIP problems.
		The first technique, called \emph{VI aggregation}, aggregates multiple VIs into a single inequality by using implications between a variable and a set of binary variables that form a clique.
		This aggregation can reduce the number of constraints and tighten the linear programming relaxation.
		The second technique, called \emph{VI-aware linear constraint propagation (LCP)}, builds on the standard LCP but incorporates VIs associated with the variable being tightened to derive more reductions and can derive tighter variable bounds.
		We show that although VI information is additionally considered, the tightest lower or upper bound of a variable can still be derived in linear time.
		Moreover, compared with a state-of-the-art approach in the literature, the proposed VI-aware LCP can derive tighter variable bounds.		
		Computational results on MIPLIB 2017 benchmark instances demonstrate the effectiveness of VI aggregation and VI-aware LCP in improving the performance of the open-source MIP solver HiGHS.
		In particular, using the two proposed presolve techniques, a reduction of 4\% in solving time and 6\% in node number on HiGHS can be achieved.
		
		
	\end{abstract}
	
	\input{section_introduction.tex}

	\input{section_notation.tex}

	\input{section_aggregation.tex}

	\input{section_boundtightening.tex}

	\input{section_computational.tex}

	\input{section_conclusion.tex}
	
	\bibliographystyle{abbrvnat}
	\bibliography{shorttitles,library}
	
	\appendix
	\renewcommand{\thesection}{Appendix~\Alph{section}}
	\renewcommand{\thesubsection}{\Alph{section}.\arabic{subsection}}
	
	\input{section_appendix.tex}

\end{document}

%% file: article_preamble.tex
\newtheorem{assumption}{Assumption}

\usepackage[letterpaper,margin=1in]{geometry}
\usepackage{setspace}
\usepackage{multicol}
\usepackage{xspace}
\usepackage{algpseudocode}
\usepackage{algorithm}
\usepackage{todonotes}

\newcommand{\MIP}{\text{MIP}\xspace}
\newcommand{\LCP}{\text{LCP}\xspace}
\newcommand{\VI}{\text{VI}\xspace}
\newcommand{\VIs}{\text{VIs}\xspace}

\newcommand{\rev}[1]{{\color{black}#1}}

\newcommand{\I}{\mathcal{I}}

\newcommand{\B}{\mathcal{B}}
\newcommand{\C}{\mathcal{C}}
\newcommand{\CO}{\mathcal{O}}
\newcommand{\CS}{\mathcal{S}}
\newcommand{\D}{\mathcal{D}}
\newcommand{\M}{\mathcal{M}}
\newcommand{\N}{\mathcal{N}}

\newcommand{\X}{\mathcal{X}}
\newcommand{\XLP}{\mathcal{X}_{\text{L}}}
\newcommand{\K}{\mathcal{K}}
\newcommand{\G}{G}
\newcommand{\V}{\mathcal{V}}

\newcommand{\LB}{\text{LB}\xspace}
\newcommand{\UB}{\text{UB}\xspace}
\newcommand{\LBhat}{\widehat{\text{LB}}\xspace}
\newcommand{\UBhat}{\widehat{\text{UB}}\xspace}

\newcommand{\HiGHS}{{{HiGHS}}\xspace}

\newcommand{\Default}{\texttt{Default}\xspace}

\newcommand{\vbca}{\texttt{VIA}\xspace}
\newcommand{\scons}{\texttt{VILCP$^\prime$}\xspace}
\newcommand{\tobj}{\texttt{VILCP}\xspace}
\newcommand{\all}{\texttt{All}\xspace}
\newcommand{\settingAchterberg}{\citet{Achterberg2013a}\xspace}

\newcommand{\tblT}{\texttt{T}\xspace}
\newcommand{\tblN}{\texttt{N}\xspace}
\newcommand{\tblIns}{\texttt{\#Ins}\xspace}
\newcommand{\tblS}{\#\texttt{S}\xspace}
\newcommand{\tblTa}{\texttt{T}$_{\vbca}$\xspace}
\newcommand{\tblTb}{\texttt{T}$_{\tobj}$\xspace}

\usepackage{accents}
\newcommand{\dunderbar}{\underaccent{\bar}{d}}

\algnewcommand{\IfThenElse}[3]{
  \State \algorithmicif\ #1\ \algorithmicthen\ #2\ \algorithmicelse\ #3}
\algnewcommand{\IfThen}[2]{
  \State \algorithmicif\ #1\ \algorithmicthen\ #2}

\newcommand{\commentInline}[1]{\item[] /*~\texttt{#1}~*/}

%% file: section_introduction.tex
\section{Introduction}\label{section-introduction}

Presolve for mixed integer programming (\MIP) problems is a set of routines applied to eliminate redundant information, strengthen the formulation, and extract information from the problem that can be used to improve the efficiency of the subsequent branch-and-cut process.
Presolve can be very effective; according to the investigations in \citet{Achterberg2020} and \citet{Achterberg2013b}, presolve can frequently make the difference between a problem being intractable and easily solvable.
Therefore, presolve has become one of the most important ingredients contributing to the power of modern MIP solvers.

A wide range of presolve techniques have been developed in the literature.
One of the earliest contributions was made by \citet{Brearley1975a}, who presented linear constraint propagation (\LCP) to tighten variable bounds.
Presolve techniques for 0-1 integer programming problems were studied by \citet{Johnson1980,Guignard1981,Crowder1983}, and \citet{Hoffman1991}.
For general \MIP problems, \citet{Savelsbergh1994} investigated various presolve techniques, among which probing  tentatively sets some binary variables to 0 or 1, applies \LCP to derive better variable bounds, and extracts useful information such as stronger variable implications (\VIs) and better global variable bounds.
Recently, new approaches to enhance probing were developed by \citet{VonHolly-Ponientzietz2026} and \citet{Dai2026}.
\citet{Gamrath2015} and \citet{Gemander2020a} investigated three presolve techniques and five two-row/column presolve techniques, respectively, and discussed their computational impact on the open-source solver SCIP \citep{Achterberg2007}.
For linear programming (LP) problems, \citet{Andersen1995} published the results on presolve and postsolve procedures;
\citet{Zhang2026} proposed to use Fourier-Motzkin elimination to simplify LP problems.
For a comprehensive discussion of various presolve techniques for \MIP problems, we refer to \citep{Achterberg2020}. 
Recently, \citet{Gleixner2023}, \citet{Hoen2024}, and \citet{Hoen2026} introduced the solver-independent presolve library PaPILO for \MIP and LP problems,  which supports parallelization, multi-precision arithmetic, certification for presolve techniques (on 0-1 integer programming problems), and delta debugging.
Details on implementing presolve techniques within an MIP solver are presented in \citet{Suhl1994}, \citet{Achterberg2007}, and \citet{Weninger2016}.

While many presolve techniques primarily aim to reduce the problem size or strengthen the formulation, other techniques (such as probing) also uncover structural information that can be reused by later components of the branch-and-cut process.
An important type of such information is the \VIs.
A \VI is a constraint of the form $y \star ax+b$, where $a, b \in \mathbb{R}$ with $a \neq 0$, $x \in \{0,1\}$, and $\star \in \{\le, \ge\}$.
It
expresses the dependency of one bound of a variable on a bound of another binary variable \citep{Achterberg2007, Maher2017}.
Therefore, \VIs arise naturally from linear constraints involving exactly two variables with at least one of them being binary.
In addition, \VIs can also be detected implicitly from more general constraints \citep{Achterberg2007} or  identified by presolve techniques like probing.
Such relations can be used to, for example, enhance the cut separation \citep{Marchand2001,Bestuzheva2021}, guide the variable branching \citep{Achterberg2007}, and improve the primal heuristics \citep{Corduk2025,Gamrath2019,Salvagnin2025} in the subsequent branch-and-cut process.
To make \VIs available to these components, modern \MIP solvers maintain them in dedicated global data structures.
In particular, \VIs between binary variables are usually stored in the clique table, and \VIs between binary and non-binary variables are usually stored in the \VI graph \citep{Achterberg2007}.


Given the important role of \VIs in various components of modern MIP solvers, there are comparatively few works on their use in presolve.
{\citet{Atamturk2000} proposed to use \VIs to enhance the bound and coefficient strengthenings, and to extract more \VIs for 0-1 integer programming problems.}
\citet{Achterberg2020} proposed to use clique information to deduce redundant constraints in a combinatorial way.
{In the same paper, the authors also discussed a procedure called clique merging/strengthening, which aims to combine several clique constraints into a single inequality; see also \cite{Brito2021}}.
Besides, cliques can also be used to derive tighter variable bounds by exploiting the fact that at most one variable in a clique can be set to one \citep{Achterberg2020,Gemander2020a}.
Note that the above techniques focus only on the \VIs between binary variables; presolve techniques utilizing the \VIs between binary and non-binary variables are relatively rare.
One exception is the presolve technique proposed in \citet{Achterberg2013a} where tighter bounds for binary and non-binary variables can be derived from knapsack constraints and \VIs.
This is equivalent to applying a simplified version of probing (where both types of \VIs are used) on the binary variables of the cover implied by a knapsack constraint, and observing that at least one of them must be set to zero.
The computational results of \citet{Achterberg2013a} showed that this presolve technique can affect a substantial number of \MIP problems in their testbed and improve the performance of the \MIP solver CPLEX.
\subsection{Contributions and outline}

{The main motivation of this paper is to develop new presolve techniques that exploit \VIs including not only those between two binary variables but also those between binary and non-binary variables.}
In particular, two such presolve techniques are proposed to enhance the capability of \MIP solvers:

\begin{itemize}
\item [$\bullet$] The first technique, called \emph{\VI aggregation}, attempts to aggregate multiple \VIs into a single inequality using the clique information.
The basic idea is that if the lower or upper bound of a variable can be tightened by fixing different binary variables through \VIs
and these binary variables form a clique, then we can aggregate the corresponding \VIs into a single inequality valid for the problem.
This can not only tighten the LP relaxation but also open up new opportunities to  reduce the size of the problem.
The proposed \VI aggregation extends the standard clique merging  \citep{Achterberg2020,Brito2021}, where reductions on \VIs involving non-binary variables are additionally performed.
\item [$\bullet$] The second technique,  called \emph{\VI-aware \LCP}, builds on the standard \LCP but involves the \VIs to derive more presolve reductions.
Specifically, the standard \LCP tightens the bounds of a variable using variable bounds, a linear constraint, and integrality constraints. In contrast, our proposed \VI-aware \LCP additionally incorporates  \VIs associated with that variable, enabling the derivation of tighter bounds.
We show that although \VIs are additionally considered, deriving the tightest lower or upper bound  for a variable can still be conducted in linear time, allowing the VI-aware LCP to be embedded into an \MIP solver.
We analyze the relation of the proposed  \VI-aware \LCP to the state-of-the-art approach in \citet{Achterberg2013a}, and demonstrate that the bounds derived by our proposed \VI-aware \LCP could be tighter.
Moreover, we also develop two enhancements to further improve the effectiveness of the proposed \VI-aware \LCP.
%
%
\end{itemize}


We implement the \VI aggregation and \VI-aware \LCP in the open-source MIP solver HiGHS \citep{Huangfu2018}. 
Computational results on MIPLIB 2017 benchmark instances demonstrate the advantage of the proposed  \VI-aware \LCP  over the state-of-the-art approach in \cite{Achterberg2013a} and 
the effectiveness of the two proposed presolve techniques---\VI aggregation and \VI-aware \LCP---in improving the performance of HiGHS.
In particular, using the two proposed presolve techniques, a reduction of 4\% in solving time and 6\% in node number on \HiGHS can be achieved.

It deserves to be mentioned that, although the proposed {\VI aggregation} and \VI-aware \LCP are designed to tackle \MIP problems, they can also be
applied to solving mixed integer nonlinear programming problems, as long as linear inequalities are presented (obtained from constraints or via cutting plane techniques) and \VIs are maintained
in dedicated data structures.

The rest of the article is organized as follows.
\cref{section-notation} introduces the notation used throughout this paper.
\cref{section-aggregation,section-boundtightening} develop the {\VI aggregation} and \VI-aware \LCP, respectively.
\cref{section-computational} presents the numerical results.
Finally, \cref{section-conclusion} concludes the paper.

%% file: section_notation.tex
\section{Notation and background}\label{section-notation}

We use the following notation throughout the paper. 
Given a matrix $A\in\mathbb{R}^{m\times n}$, vectors 
$c,~ u\in \mathbb{R}^n$, $b \in \mathbb{R}^m$, 
and variables $x\in$ $\mathbb{R}^n$ with $x_j\in\mathbb{Z}$ for $j\in \I\subseteq \N := \left\{1,\cdots,n\right\}$, 
the \MIP problem can be written as
\begin{equation}
	\min \left\{c^\top x ~\mid~ Ax \le b,~ 0 \le x \le u,~ {x_j \in \mathbb{Z},~ \forall~j \in \I}\right\}.
	\label{MIP}
	\tag{MIP}
\end{equation}
For simplicity of discussion, we assume throughout the paper that the lower bounds are all zero, and the upper bounds $u_j$ for $j \in\N$ are all finite; however, many of our results can be generalized to cases where some of the lower bounds are not zero, or some of the upper bounds are infinite.
We denote the index set of constraints as $\M = \left\{1,\cdots,m\right\}$ and the feasible region of \eqref{MIP} as 
\begin{equation}\label{Xdef}
\X= \{ x\in \mathbb{R}^n ~ \mid ~ Ax \le b,~ 0 \le x \le u,~ {x_j \in \mathbb{Z},~ \forall~j \in \I} \}.
\end{equation}
The linear relaxation of $\X$ is denoted as $\XLP:=\{ x\in \mathbb{R}^n ~ \mid ~ Ax \le b,~ 0 \le x \le u\}$, obtained by dropping the integrality constraints on $x_j$ for $j \in \I$ from $\X$.
Given a vector $a \in \mathbb{R}^n$ and a subset $\CS \subseteq \N$, we define $a_{\CS}$ to be the vector $a$ restricted to indices in $\CS$.
Finally, we use $\B = \left\{j \in \I ~\mid~ u_j = 1\right\}$ to denote the index set of binary variables.

In this paper, we consider \VIs of the form:
\begin{subequations}
	\begin{multicols}{2}
		\noindent
		\begin{align}
			& x_i = 0 \rightarrow x_j \geq \ell_{ij}, ~\forall~ j \in \N^\geq_0(i),~ i \in \B,\label{impa1} \\
			& x_i = 1 \rightarrow x_j \geq \ell_{ij}, ~\forall~ j \in \N^\geq_1(i),~ i \in \B,\label{impa2}
		\end{align}
		\begin{align}
			& x_i = 0 \rightarrow x_j \leq u_{ij},    ~\forall~ j \in \N^\leq_0(i),~ i \in \B,\label{impa3} \\
			& x_i = 1 \rightarrow x_j \leq u_{ij},    ~\forall~ j \in \N^\leq_1(i),~ i \in \B,\label{impa4}
		\end{align} 
	\end{multicols}
\end{subequations}\noindent
where for each $i \in \B$, $\N_v^{\star}(i)\subseteq \N$ with $\star \in \{\ge , \le \}$ denotes the set of variables from which implied lower bounds  $\ell_{ij}$ or upper bounds  $u_{ij}$ can be derived by setting $x_i = v \in \left\{0,1\right\}$.
Throughout, we pose the following trivial assumption on the \VIs defined in \eqref{impa1}--\eqref{impa4}.
\begin{assumption}\label{assum1}
	(i) \VIs \rev{~in} \eqref{impa1}--\eqref{impa4} are non-redundant, that is,  $0 < \ell_{ij} \le u_j$ and $0 \le u_{ij} <  u_j$.
	\item[(ii)] For the \VI $x_i = v \rightarrow x_j \ge \ell_{ij}$ (respectively, $x_i = v \rightarrow x_j \le u_{ij}$) with $j \in \I$, we can replace $\ell_{ij}$ (respectively, $u_{ij}$) with $\lceil \ell_{ij} \rceil$ (respectively, $\lfloor u_{ij} \rfloor$) to obtain a tighter implied bound.
	Therefore, we can assume $\ell_{ij} \in \mathbb{Z}$ for all $i \in \B$ and $j \in \left(\N^\geq_0(i) \cup  \N^\geq_1(i)\right) \cap \I$, and $u_{ij} \in \mathbb{Z}$ for all $i \in \B$ and $j \in \left(\N^\leq_0(i) \cup  \N^\leq_1(i)\right) \cap \I$.

    \item[(iii)] Self \VIs are not considered, that is, $i \notin \N_v^{\star}(i)$ holds for any $\star \in \{\ge , \le \}$ and $v \in \{0,1\}$.
	\item[(iv)] 
	For two \VIs $x_i = 0 \rightarrow x_j \geq a$ and $x_i = 1 \rightarrow x_j \geq b$, we can replace them with a (new) lower bound constraint $x_j \geq \min\{a,b\}$; 
	for two \VIs $x_i = 0 \rightarrow x_j \leq a$ and $x_i = 1 \rightarrow x_j \leq b$, we can replace them with a (new) upper bound constraint $x_j \leq \max\{a,b\}$. 
	Therefore, we can assume $\N_0^\ge (i) \cap \N_1^\ge (i) =\varnothing $ and $\N_0^\le (i) \cap \N_1^\le (i) = \varnothing$ for all $i \in \B$.
	\item[(v)] For two \VIs $x_i = v \rightarrow x_j \ge a$ and $x_i = v \rightarrow x_j \le b$ with $v \in \{0,1\}$, if $a > b$, we can fix $x_i = 1-v$ (as $x_i = v$ makes an infeasible problem).
	Therefore, we can assume $\ell_{ij} \le u_{ij}$ for all $i \in \B$ and $j \in \left(\N_0^\ge (i) \cap \N_0^\le (i)\right) \cup \left(\N_1^\ge (i) \cap \N_1^\le (i)\right)$.
\end{assumption}

Note that the \VIs in \eqref{impa1}--\eqref{impa4} can be equivalently expressed as the following inequalities:
\begin{subequations}
	\begin{multicols}{2}
		\noindent
		\begin{align}
			& x_j \ge - \ell_{ij} x_i + \ell_{ij}, ~\forall~ j \in \N^\geq_0(i),~ i \in \B,\label{impc1} \\
			& x_j \ge \ell_{ij} x_i, ~\forall~ j \in \N^\geq_1(i),~ i \in \B,\label{impc2}
		\end{align}
		\begin{align}
			& x_j \le \left(u_{j} - u_{ij}\right) x_i + u_{ij},~\forall~ j \in \N^\leq_0(i),~ i \in \B,\label{impc3} \\
			& x_j \le \left(u_{ij} - u_j\right) x_i + u_j, ~\forall~ j \in \N^\leq_1(i),~ i \in \B. \label{impc4}
		\end{align} 
	\end{multicols}
\end{subequations}\noindent
This equivalence also implies that \VIs can be derived from constraints  of problem \eqref{MIP} involving exactly two variables with at least one being binary.
In addition to this, \VIs can also be derived from presolve techniques such as probing \citep{Savelsbergh1994}. 
In the following, we will use the \VIs in \eqref{impa1}--\eqref{impa4}  and \eqref{impc1}--\eqref{impc4} interchangeably. 

Depending on whether $j \in \B$, \VIs of the form \eqref{impa1}--\eqref{impa4} can be categorized into two types.
In particular, if $j \in \N \backslash \B$, they represent \VIs between binary and non-binary variables, which are stored in a data structure called \VI graph $G$ \citep{Achterberg2007,Achterberg2013a}.
If $j \in \B$, they represent \VIs between binary variables and therefore, $\ell_{ij} = 1$ and $u_{ij} = 0$.
In such cases, they are aggregated and stored in a data structure called clique table $\mathfrak{C}$, where each $\C=(\C^+, \C^-)\in \mathfrak{C}$ corresponds to the clique inequality $\sum_{i \in \C^+} x_i + \sum_{i \in \C^-}(1-x_i) \leq 1$ \citep{Atamturk2000,Achterberg2007}.
This structure avoids storing $\frac{|\C|(|\C|-1)}{2}$ pairwise \VIs: $(x_i = v_i) \rightarrow (x_j = w_j)$ for $i,~ j \in \C$ with $i \ne j$ (where $v_i =1$ for $i \in \C^+$ and $v_i =0$ otherwise; $w_j =0$ for $j \in \C^+$ and $w_j =1$ otherwise). 
Note that in many MIP solvers (e.g., HiGHS), \VIs between non-binary variables are not maintained.


Using \eqref{impa1}--\eqref{impa4}, one can easily access variables whose bounds can be tightened after fixing $x_i$.
Conversely, for a variable $x_j,~ j \in \N$, we can access binary variables that can derive tighter bounds for $x_j$ as follows:
\begin{subequations}
	\begin{multicols}{2}
		\noindent
		\begin{align}
			& x_j < \ell_{ij} \rightarrow x_i = 1,~\forall ~ i \in \B^{<}_1(j), \label{impb4}\\
			& x_j < \ell_{ij} \rightarrow x_i = 0,~\forall ~ i \in \B^{<}_0(j), \label{impb3}
		\end{align}
		\begin{align}
			& x_j > u_{ij} \rightarrow x_i = 1,~\forall ~i \in \B^{>}_1(j),     \label{impb2}\\
			& x_j > u_{ij} \rightarrow x_i = 0,~ \forall~i \in \B^{>}_0(j),     \label{impb1}
		\end{align} 
	\end{multicols}
\end{subequations}\noindent
where
\begin{subequations}
	\begin{multicols}{2}
        \noindent
        \begin{align}
			&\B_1^{<}(j) := \{ i \in \B \mid j \in \N_0^{\geq}(i) \}, \label{b1}\\
			&\B_0^{<}(j) := \{ i \in \B \mid j \in \N_1^{\geq}(i) \},\label{b2}
		\end{align}
        \begin{align}
			&\B_1^{>}(j) := \{ i \in \B \mid j \in \N_0^{\leq}(i) \},\label{b3} \\
    		&\B_0^{>}(j) := \{ i \in \B \mid j \in \N_1^{\leq}(i) \}.\label{b4} 
		\end{align} 
    \end{multicols}
\end{subequations}\noindent
From (iii), (iv), and (v) of \cref{assum1}, it must follow   
\begin{remark}\label{remark1}
		For $j \in \N$, it follows that (i) $j \notin  \B_v^{\star}(j)$ holds for any $\star \in \{< , > \}$ and $v \in \{0,1\}$; 
		(ii) $\B_0^< (j) \cap \B_1^< (j) = \varnothing$ and $\B_0^> (j) \cap \B_1^> (j) =\varnothing$; and
		(iii) $\ell_{ij} \le u_{ij}$ holds for all $i \in \left(\B_1^<(j) \cap \B_1^>(j)\right) \cup \left(\B_0^<(j) \cap \B_0^>(j)\right)$.
\end{remark}



%% file: section_aggregation.tex
\section{Variable implication aggregation}\label{section-aggregation}

In this section, we propose the \VI aggregation where
the basic idea is that, if the lower bound (upper bound, respectively) of variable $x_j$, $j \in \N$, can be tightened by fixing different binary variables via \VIs in \eqref{impc1}--\eqref{impc2} (\eqref{impc3}--\eqref{impc4}, respectively), and these binary variables form a clique, then we can aggregate such \VIs into a single inequality (called aggregated \VI), which is valid for \eqref{MIP} and dominates all individual \VIs being aggregated, thereby yielding a tighter LP relaxation.
If the \VIs being aggregated appear as constraints in \eqref{MIP}, they can be replaced by the resulting aggregated \VIs, leading to a reduction of the total number of constraints.

\subsection{Deriving aggregated implications}

We start with the following illustrative example. 
\begin{example}\label{ex-agg}
	Consider the following mixed integer set $\X \subseteq \mathbb{R}^5$:
	\begin{equation*}
		\X := \left\{\left(x_1,~ x_2,~ x_3,~ x_4~, x_5\right) ~\mid~ \eqref{exvb-1}\text{--}\eqref{exvb-5}\right\},
	\end{equation*}
	where 
	\begin{subequations}
		\begin{multicols}{2}
			\noindent
			\begin{align}
				\label{exvb-1}
				&x_5 \ge x_1, \\
				\label{exvb-2}
				&x_5 \ge 2x_2, \\
				\label{exvb-3}
				&x_5 \ge 3x_3, \\
				\label{exvb-4}
				&x_5 \ge 4x_4,
			\end{align}
			\begin{align}
				\label{exclique-1}
				&x_1 + x_2 \le 1, \\
				\label{exclique-2}
				&x_2 + x_3 \le 1, \\
				\label{exclique-3}
				&x_1 + x_4 \le 1, \\
				\label{exvb-5}		
				&x_i \in \{0,1\},~ i = 1, 2, 3, 4, ~ x_5 \in [0, +\infty).
			\end{align} 
		\end{multicols}
	\end{subequations}\noindent
	From the clique inequality \eqref{exclique-1}, fixing any single binary variable of $x_1,x_2$ to one immediately forces the other variable to be zero.
	Therefore, we can aggregate the \VIs in \eqref{exvb-1}--\eqref{exvb-2} into a single inequality:
	\begin{equation}\label{exagg-cons}
		x_5 \ge x_1 + 2x_2,
	\end{equation}
	which is valid for the mixed integer set $\X$.
	After adding \eqref{exagg-cons} into $\X$, the \VIs in \eqref{exvb-1}--\eqref{exvb-2} are dominated by \eqref{exagg-cons} and therefore can be removed from $\X$; that is, 
	\begin{equation}\label{ex-xprime}
		\X=\X^\prime := \left\{\left(x_1,~ x_2,~ x_3,~ x_4~, x_5\right) ~\mid~ \eqref{exvb-3}\text{--}\eqref{exvb-5},~ \eqref{exagg-cons}\right\}.
	\end{equation}
	Note that the number of constraints used to characterize $\X^\prime$ is smaller than that used to characterize $\X$.
	Moreover, adding inequality \eqref{exagg-cons} leads to a tighter linear relaxation, i.e., $\XLP^\prime \subsetneq \XLP$.
	Indeed, as \eqref{exagg-cons} implies \eqref{exvb-1}--\eqref{exvb-2}, $\XLP^\prime \subseteq \XLP$;
	on the other hand,  $\hat{x} := \left(\frac{1}{2},~ \frac{1}{2},~0,~0,~ 1\right)  \in \XLP$ but  $\hat{x} \notin \XLP^\prime$. 
	Similarly, we can explore the clique information \eqref{exclique-2}--\eqref{exclique-3}, and aggregate constraints in \eqref{exvb-1}--\eqref{exvb-4} as $x_5 \ge 2 x_2 + 3x_3$ and $x_5 \ge x_1 + 4x_4$. 
\end{example}

The above procedure of aggregating \VIs is formalized below.
Without loss of generality, for a given variable $x_j,~ j \in \N$, we consider the aggregation of the \VIs in \eqref{impb3}, which can be written as:
\begin{equation}\label{aggsystem}
	x_j \geq \ell_{ij} x_i,~\forall~i \in \B_0^<(j).
\end{equation}
For a clique constraint $\sum_{i \in \C} x_i \leq 1$ with $\C \subseteq \B_0^<(j)$, it is straightforward that the \VIs $	x_j \geq \ell_{ij} x_i$ with $i \in \C$ in \eqref{aggsystem} can be aggregated into a new inequality (called aggregated \VI) $x_j \geq \sum_{i \in \C} \ell_{ij} x_i$, which dominates all individual \VIs $x_j \geq \ell_{ij} x_i$ with $i \in \C$.
Note that for \VIs in \eqref{impb4}, by complementing $x_i$ as $\bar{x}_i:=1-x_i$, we obtain
\begin{equation}\label{tmptmp}
	x_j \ge - \ell_{ij} x_i + \ell_{ij} ~\iff~ x_j \ge \ell_{ij} \bar{x}_i,~ \forall~ i \in \B_1^<(j).
\end{equation}
Therefore, when considering aggregation of \VIs, we also consider those in \eqref{tmptmp}, which enables us to construct a stronger aggregated inequality:
\begin{equation}\label{agg-single}
	x_j	\ge \sum_{i \in \C^+} \ell_{ij} x_i +  \sum_{i \in \C^-} \ell_{ij} \bar{x}_i,  
\end{equation}
where $(\C^+, \C^-)$ is a clique with $\C^+ \subseteq  \B_0^<(j)$ and $\C^-\subseteq \B_1^<(j)$. 
Note that since $\B_0^< (j) \cap \B_1^< (j) = \varnothing$ (as shown in (ii) of \cref{remark1}), it follow $\C^+\cap \C^-=\varnothing$.
Also note that for \VIs in \eqref{impb2}--\eqref{impb1} (where the upper bound of $x_j$ is considered), using the complement $\bar{x}_j = u_j - x_j$, we can transform them into \VIs of the form  \eqref{impb4}--\eqref{impb3} and therefore, we can derive similar aggregated \VIs.

Three remarks on the \VI aggregation are in order.
First, some of the \VIs in \eqref{aggsystem}--\eqref{tmptmp} appear as constraints in \eqref{MIP}, while others are detected from the whole problem by presolve techniques (such as probing) and do not appear as constraints in \eqref{MIP}.
Therefore, the number of constraints in \eqref{MIP} will be reduced only if some \VIs $x_j \geq \ell_{ij} x_i$ for $i \in \C^+$ or $x_j \geq \ell_{ij} \bar{x}_i$ for $i \in \C^-$ appear as constraints in  \eqref{MIP}. 
In our implementation, the aggregated \VI in \eqref{agg-single} is added only when at least one \VI appears as a constraint in \eqref{MIP}.
Second, if the clique  $\C$ is not maximal (i.e., there exists \rev{a clique} such that $\bar{\C}\subseteq  \B_0^<(j)\cup \B_1^<(j)$ and $\C \subsetneq \bar{\C}$ hold), then inequality \eqref{agg-single} is dominated by the inequality constructed using the maximal clique $\bar{\C}$.
Therefore, in our implementation, we only consider the case  that $\C$ is maximal.
Third, it is interesting to observe that if $x_j \in \{0,1\}$ in \eqref{agg-single}, then $\ell_{ij} = 1$; therefore, inequality \eqref{agg-single} reduces to
\begin{equation}
	\sum_{i \in \C^+} x_i  +\sum_{i \in \C^-} \bar{x}_i \le x_j\Longleftrightarrow\sum_{i \in \C^+} x_i + \sum_{i \in \C^- \cup \{j\} } \bar{x}_i \le 1, 
\end{equation}
which is a (strengthened) clique inequality that involves one more binary variable in the underlying clique.
This procedure of extending an existing clique is called \textit{clique merging/strengthening} \citep{Achterberg2020,Brito2021}.
Therefore, our proposed \VI aggregation can be viewed as a generalization of the clique \rev{merging} that additionally considers reduction on non-binary variables.

\subsection{Implementation}
 
Next, we attempt to derive a small subset of maximal cliques such that the corresponding aggregated \VIs can imply all \VIs in \eqref{aggsystem}--\eqref{tmptmp}.
To achieve this, we need the concept of clique covers introduced below.

\begin{definition}\label{cliquecover}
	A clique cover of an index set of binary variables $\CS$, denoted as $\mathfrak{S}$, is a subset of $\mathfrak{C}$ satisfying $\mathcal{S} = \cup_{\C \in \mathfrak{S}} \C$ (in other words, for every index $i \in \mathcal{S}$, there exists some clique $\C \in \mathfrak{S}$ such that $i \in \C$).
	It is minimal if $\mathcal{S} \neq \cup_{\C \in \mathfrak{S} \backslash \left\{\C^\prime\right\}} \C$ holds for any $\C^\prime \in \mathfrak{S}$.
\end{definition}
\noindent Letting $\mathfrak{S} = \{\C_1, \ldots, \C_q \}$ be a clique cover of $\B_0^<(j)\cup \B_1^<(j)$ (where $q\in \mathbb{Z}_+$), all \VIs in \eqref{aggsystem}--\eqref{tmptmp} can be aggregated into $q$ inequalities:
\begin{equation}\label{impl-aggregated}
	x_j \geq \sum_{i \in \C^+_p} \ell_{ij} x_i +  \sum_{i \in \C^-_p} \ell_{ij} (1-x_i)   , ~\forall~p \in \left\{1,\ldots,q \right\}.
\end{equation}
Here, the $p$-th aggregated \VI (i.e., the inequality constructed using $\C_p$ in \eqref{impl-aggregated}) implies the following \VIs: 
\begin{equation*}
	x_j \ge \ell_{ij}x_i , ~\forall~ i \in \C^+_p,~x_j \ge \ell_{ij}(1-x_i) , ~\forall~ i \in \C^-_p.
\end{equation*}

Note that in general, given the clique table $\mathfrak{C}$, there may exist different minimal clique covers of $\B_0^<(j)\cup \B_1^<(j)$; finding the one with minimum number of cliques could be computationally expensive.
Therefore, we instead use the heuristic approach summarized in \cref{alg:clique-cover} to construct a minimal clique cover of $\B_0^<(j)\cup \B_1^<(j)$. 
Specifically, in step 3 of  \cref{alg:clique-cover}, we initialize $\K \gets \B_0^<(j)\cup \B_1^<(j)$ as the set of uncovered indices, and $\mathfrak{S} = \varnothing$ as the clique cover.
In steps 4--7 of \cref{alg:clique-cover}, we find a  clique $\C $ in $\mathfrak{C}$ that maximizes the number of uncovered indices $|\C\cap\K|$.
Finally, in steps 8--12, we remove some redundant cliques to ensure that the resulting clique cover is minimal.

\begin{algorithm}[t]
	\footnotesize
	\caption{A procedure to construct a minimal clique cover of $\B_0^<(j)\cup \B_1^<(j)$}
	\label{alg:clique-cover}
	\begin{algorithmic}[1]
		\State \textbf{Input:} Variable index $j \in \N$, binary variable set $\B_0^<(j)\cup \B_1^<(j)$, and clique table $\mathfrak{C}$.
		\State \textbf{Output:} A clique cover $\mathfrak{S}$ of $\B_0^<(j)\cup \B_1^<(j)$.
		\State Initialize $\K \gets \B_0^<(j)\cup \B_1^<(j)$ and $ \mathfrak{S} = \varnothing$.
		\While{$\K \neq \varnothing$}
			\State Find a clique $\C \in \mathfrak{C}$ that  maximizes $\left|\C \cap \K\right|$. 
			\State Set $\mathfrak{S} \gets \mathfrak{S} \cup \left\{{\C \cap \K}\right\}$ and $\K \gets \K \backslash {\C}$.
		\EndWhile

		\For{each $\mathcal{C}\in\mathfrak{S}$}
		\If{$ \bigcup_{\mathcal{C}'\in\mathfrak{S}\setminus\{\mathcal{C}\}}\mathcal{C}'=\B_0^<(j)\cup \B_1^<(j)$}
		\State $\mathfrak{S} \gets \mathfrak{S} \setminus \{\mathcal{C}\}$.
		\EndIf
		\EndFor
		
	\end{algorithmic}
\end{algorithm}

\begin{example}
	Consider the mixed integer set $\X$ defined by \eqref{exvb-1}--\eqref{exvb-5} in \cref{ex-agg}, where the set of maximal cliques is $\mathfrak{C}= \{ \{1,2\}, ~\{2,3\},~\{1,4\} \}$. 
	We want to derive a minimal clique cover of the set $\B_0^<(5) = \left\{1,2,3,4\right\}$.
	Applying \cref{alg:clique-cover}, we obtain a minimal clique cover $\mathfrak{S} = \left\{\left\{2,3\right\}\right.$, $\left.\left\{1,4\right\}\right\}$.
	We can derive the aggregated \VIs $x_5 \ge 2 x_2 + 3x_3$ and $x_5 \ge x_1 + 4x_4$ by utilizing the two maximal cliques, which imply all \VIs in \eqref{exvb-1}--\eqref{exvb-4}.
\end{example}

In order to remove more \VI constraints  from \eqref{MIP}, we use a two-phase approach to construct the 
minimal cover of $\B_0^<(j)\cup \B_1^<(j)$.
Specifically, in the first phase, we collect \VIs in \eqref{aggsystem}--\eqref{tmptmp} that appear as constraints in problem \eqref{MIP} as index sets $\bar{\B}_0^<(j) \subseteq \B_0^<(j)$ and $\bar{\B}_1^<(j) \subseteq \B_1^<(j)$ (i.e., $i \in \bar{\B}_0^<(j)$ or $i \in \bar{\B}_1^<(j)$ implies that $x_j \ge \ell_{ij} x_i $ or $x_j \ge \ell_{ij} \left(1 - x_i\right) $ arises as a constraint in problem \eqref{MIP}), and 
use \cref{alg:clique-cover} to determine 
 a minimal clique cover ${\mathfrak{S}} = \left\{\C_1,\ldots,\C_q\right\}$ of $\bar{\B}_0^<(j) \cup \bar{\B}_1^<(j)$.
In the second phase, we attempt to insert the indices in $\B_0^<(j) \backslash \bar{\B}_0^<(j)$ and $\B_1^<(j) \backslash \bar{\B}_1^<(j)$ into existing cliques in $\mathfrak{S}$, so as to strengthen the aggregated \VIs. 
Notice that this two-phase implementation guarantees that  each aggregated \VI implies at least one distinct \VI constraint in \eqref{MIP}, thereby avoiding increasing the number of constraints in the original problem.

%% file: section_boundtightening.tex
\section{Variable implication-aware linear constraint propagation}\label{section-boundtightening}

In this section, we first briefly review the \LCP and discuss its weakness in tightening variable bounds---the underlying optimization problems for computing the bounds of a variable only consider variable bounds, a linear constraint, and integrality constraints.
Then, we propose to enhance \LCP with \VIs where the underlying optimization problems for computing the bounds of a variable further consider the \VIs associated with that variable.
We show that the new optimization problems can still be solved in linear time.
Subsequently, we analyze the relation of the proposed \VI-aware \LCP to the state-of-the-art approach in \citet{Achterberg2013a}.
Finally, we provide the implementation details and two enhancements of the proposed \VI-aware \LCP.

\subsection{Review of linear constraint propagation}\label{classicdp}


We first review the \LCP, which is now one of the key ingredients in all modern MIP solvers \citep{Achterberg2007,Achterberg2020}.
Consider a single constraint, the variable bounds, and the integrality constraints in \eqref{MIP}:
\begin{subequations}
	\begin{align}
		& a_r x_r + a_{\N \backslash \{r\}}^\top x_{\N \backslash \{r\}} \leq b_0,\label{linear}\\
		& 0 \leq x \leq u,\label{bounds}\\
		& x_{\I} \in \mathbb{Z}^{\left|\I\right|}, \label{intCons} 
	\end{align}
\end{subequations}
where we recall that $\N=\{1, \ldots n\}$, $a \in \mathbb{R}^n $, $b_0 \in \mathbb{R}$, and $0 < u_j < +\infty $ for all $j \in \N$.
For simplicity of discussion, we pose the following trivial assumption throughout this section.
\begin{assumption}\label{assum2}
	 $a_j \geq 0$ holds for all $j \in \N$.
\end{assumption}
\noindent
Indeed, if $a_j < 0$, we can complement variable $x_j$ as $\bar{x}_j = u_j - x_j$ such that the coefficient of $\bar{x}_j$ is $-a_j > 0$.
Applying the \LCP, a tighter upper bound $b_0 / a_r$ for variable $x_r$ can be obtained if  $a_r > 0$ and  $b_0 / a_r< u_r$.
If $x_r$ is an integer variable (i.e., $r \in \I$), we can round down this bound to $\lfloor b_0 / a_r \rfloor$.
The implementation of \LCP is very efficient; overall, tightening bounds for all variables can be conducted with \rev{the} complexity \rev{of} $\CO(n)$.

However, the above approach suffers from two limitations: (i) only the upper bounds of variables with non-zero coefficients in \eqref{linear} can be tightened; 
(ii) even for variables whose bounds can be tightened, the new bounds may be  weak.
To see the latter, we note that the new upper bound is equal to the optimal value of the following optimization problem:
\begin{align}
	& \max\left\{x_r ~\mid~ \eqref{linear}\text{--}\eqref{intCons}\right\},
	\label{single-prob-2}
\end{align}
where only the linear constraint \eqref{linear}, the variable bounds \eqref{bounds}, and the integrality constraints \eqref{intCons} are considered. 
In the following, we will incorporate \VIs in \eqref{impa1}--\eqref{impa4} into the \LCP so that \emph{tighter bounds for more variables can be detected}. 

\subsection{Enhancing linear constraint propagation by utilizing variable implications}\label{subsect:boundtightening-theory}

In this section, we attempt to enhance \LCP by taking \VIs in  \eqref{impa1}--\eqref{impa4} into consideration.
Our approach is to involve \VIs of $x_r$ when deriving the bounds for variable $x_r$. 
In particular, for $r \in \B$, we attempt to derive {the tightest lower and upper bounds of $x_r$ from the following system}:
\begin{subequations}
	\begin{multicols}{2}
		\noindent
		\begin{align}
			& \eqref{linear}\text{--}\eqref{intCons}, \nonumber \\
			& x_r = 0 ~\to~ x_j \ge \ell_{rj},~ \forall~ j \in \N_0^\ge(r),  \label{dp-imp-binary1} \\ 
			& x_r = 1 ~\to~ x_j \ge \ell_{rj},~ \forall~ j \in \N_1^\ge(r),  \label{dp-imp-binary2}
		\end{align}
		\begin{align}
			& \tag*{}\\
			& x_r = 0 ~\to~ x_j \leq u_{rj},    ~\forall~ j \in \N^\leq_0(r),\label{dp-imp-binary3} \\
			& x_r = 1 ~\to~ x_j \leq u_{rj},    ~\forall~ j \in \N^\leq_1(r). \label{dp-imp-binary4}
		\end{align} 
	\end{multicols}
\end{subequations}\noindent
For $r \in \N \backslash \B$, we attempt to derive the tightest lower and upper bounds of $x_r$ from the following system:
\begin{subequations}
	\begin{multicols}{2}
		\noindent
		\begin{align}
			& \eqref{linear}\text{--}\eqref{intCons}, \nonumber \\
			& x_r < \ell_{ir} ~\to~ x_i = 1, ~\forall~ i \in \B_1^<(r), \label{dp-imp-nonb1}\\
			& x_r < \ell_{ir} ~\to~ x_i = 0, ~\forall~ i \in \B_0^<(r), \label{dp-imp-nonb2}
		\end{align}
		\begin{align}
			& \tag*{}\\
			& x_r > u_{ir}    ~\to~ x_i = 1, ~\forall~ i \in \B_1^>(r), \label{dp-imp-nonb3} \\
			& x_r > u_{ir}    ~\to~ x_i = 0, ~\forall~ i \in \B_0^>(r). \label{dp-imp-nonb4} 
		\end{align} 
	\end{multicols}
\end{subequations}\noindent
In the following, we consider the two cases separately.

\subsubsection{The case $r \in \B$.}\label{case-binary}

For $r \in \B$, deriving the tightest bounds $\LB \le x_r \le \UB$ from the system \eqref{linear}--\eqref{intCons} and \eqref{dp-imp-binary1}--\eqref{dp-imp-binary4} is equivalent to solving the following two optimization problems:
\begin{subequations}
	\begin{align}
		\LB &= \min\left\{x_r ~\mid~ \eqref{linear}\text{--}\eqref{intCons},~\eqref{dp-imp-binary1}\text{--}\eqref{dp-imp-binary4} \right\}, \label{bin-prob1}\\
		\UB &= \max\left\{x_r ~\mid~ \eqref{linear}\text{--}\eqref{intCons},~\eqref{dp-imp-binary1}\text{--}\eqref{dp-imp-binary4} \right\}. \label{bin-prob2}
	\end{align}
\end{subequations}
As \eqref{dp-imp-binary1}--\eqref{dp-imp-binary4} are \VIs that can be written as linear constraints of the form \eqref{impc1}--\eqref{impc4}, problems \eqref{bin-prob1}--\eqref{bin-prob2} are MIP problems. 
However, in the context of implementing a presolve technique, using an off-the-shelf MIP solver for solving problems \eqref{bin-prob1}--\eqref{bin-prob2} could be computationally demanding.
To develop an efficient approach for computing $\LB$ and $\UB$ without the need for calling an \MIP solver, we introduce the notation of implied activity.

\begin{definition}\label{defbin}
	For linear constraint \eqref{linear}, 
	the \textit{implied activity} $w_r(v)$ with respect to $r \in \B$ and $v \in \left\{0,1\right\}$  is defined as
	\begin{equation}\label{tmeq1}
		w_r(v) := \min_{x_r =v} \left\{\sum_{j \in \N} a_j x_j ~\Bigm|~  \eqref{bounds}\text{--}\eqref{intCons},~\eqref{dp-imp-binary1}\text{--}\eqref{dp-imp-binary4} \right\}.
	\end{equation}
\end{definition}
\noindent By definition, $w_r(v) \le b_0$ holds if and only if there exists a solution ${x}$ with ${x}_r = v$  satisfying \eqref{linear}--\eqref{intCons} and \eqref{dp-imp-binary1}--\eqref{dp-imp-binary4}.
This enables us to reformulate \eqref{bin-prob1} and \eqref{bin-prob2} as the following two optimization problems:
\begin{subequations}
	\begin{align}
		\label{re-bin-0}
		\LB &= \min\left\{v \in \left\{0,1\right\} ~\mid~ w_r(v) \le b_0 \right\},\\
		\label{re-bin-1}
		\UB &= \max\left\{v \in \left\{0,1\right\} ~\mid~ w_r(v) \le b_0 \right\}.
	\end{align}
\end{subequations}
Therefore, it must follow
\begin{proposition} \label{0-1prop}
(i) $w_r(0) > b_0$ holds if and only if $\LB = 1$;
(ii) $w_r(1) > b_0$ holds if and only if $\UB = 0$.
\end{proposition}

From \cref{0-1prop}, determining the lower and upper bounds \LB and \UB for binary variable $x_r$ is equivalent to computing  $w_r(0)$ and $w_r(1)$ and checking whether $w_r(0)>b_0$ or $w_r(1)>b_0$ holds.
The computation of $w_r(0)$ and $w_r(1)$  can be done by analyzing \VIs in \eqref{dp-imp-binary1}--\eqref{dp-imp-binary4}.
In fact, for $w_r(0)$, 
it follows from $x_r = 0$, \eqref{dp-imp-binary1}, and \eqref{dp-imp-binary3} that $x_j \ge \ell_{rj}$ \rev{holds} for $j \in \N^{\ge}_0(r)$ and $x_j \le u_{rj}$ \rev{holds} for $j \in \N^\le_0(r)$.
Note that from \cref{assum1}(v), it must follow that $\ell_{rj} \le u_{rj}$ \rev{holds} for all $j \in \N_0^\ge(r) \cap \N_0^\le(r)$.
Therefore, one optimal solution $x^*$ for problem \eqref{tmeq1} is given as
\begin{equation*}
	x^*_j = \ell_{rj}, ~\forall~j \in \N_0^\ge(r), ~x^*_j = 0,~\forall~j \in \N \backslash \N_0^\ge(r).
\end{equation*}
Thus,
\begin{equation}
	w_r(0) = \sum_{\substack{j \in \N^{\ge}_0(r)}} a_j \, \ell_{rj}.	\label{w0}
\end{equation}
Similarly, we can compute
\begin{equation}\label{w1}
w_r(1) = \sum_{\substack{j \in \N^{\ge}_1(r)}} a_j \, \ell_{rj} + a_r.
\end{equation} 
Note that from \eqref{w0} and \eqref{w1}, $w_r(0)$ and $w_r(1)$ are determined only by \eqref{linear}--\eqref{intCons} and \eqref{dp-imp-binary1}--\eqref{dp-imp-binary2}, which implies that \VIs in \eqref{dp-imp-binary3}--\eqref{dp-imp-binary4} can be omitted when computing \eqref{w0} and \eqref{w1}. 
Therefore, for variable $x_r$, the overall complexity for the computation of lower and upper bounds  \eqref{bin-prob1} and \eqref{bin-prob2} is $\CO\left(\left|\N^{\ge}_0(r)\right| + \left|\N^{\ge}_1(r)\right|\right)$.

\begin{remark}\label{remark2}
	Note that if $x_r$ can be fixed by \cref{0-1prop}, then bounds for other variables can also be derived by the \VIs \rev{~in} \eqref{dp-imp-binary1}--\eqref{dp-imp-binary4}.
	Specifically, if $x_r = 0$, then from \eqref{dp-imp-binary1} and \eqref{dp-imp-binary3}, we obtain $x_j \ge \ell_{rj}$ for  $j \in \N_0^\ge(r)$ and $x_j \le u_{rj}$ for $j \in \N_0^\le(r)$;
	if $x_r = 1$, then from \eqref{dp-imp-binary2} and \eqref{dp-imp-binary4}, we obtain $x_j \ge \ell_{rj}$ for $j \in \N_1^\ge(r)$ and $x_j \le u_{rj}$ for $j \in \N_1^\le(r)$. 
\end{remark}


\begin{example}\label{example1}
	Consider the following linear constraint, variable bounds, and integrality constraints:
		\begin{align}
		 x_2 + 0.9x_3 + 0.5x_5 \le 2 ,~ x_i \in \left\{0,1\right\},~ i = 1,2,3,4,~ x_5 \in \left[0,3\right], \label{ex}
		\end{align}
	and the following \VIs:
	\begin{subequations}
		\begin{multicols}{2}
			\noindent
			\begin{align}
				&x_1 = 0 \to x_2 \ge 1, \label{e-imp1} \\
				&x_1 = 0 \to x_3 \ge 1, \label{e-imp2} \\
				&x_1 = 0 \to x_5 \ge 0.4, \label{e-imp3} \\
				&x_1 = 0 \to x_5 \le 0.5, \label{e-x4}
			\end{align}
			\begin{align}
				&x_1 = 1 \to x_4 \ge 1, \label{e-x5} \\
				&{x_5 > 0.5 \to x_2 = 0}, \label{ex-D} \\
				&x_5 < 1 \to x_2 = 1 \label{e-imp4}, \\
				&x_5 > 2 \to x_3 = 1 \label{e-imp5}.
			\end{align} 
		\end{multicols}
	\end{subequations}\noindent
	As the coefficient of variable $x_1$ is zero, the bounds of variable $x_1$ cannot be tightened by applying the standard \LCP on \eqref{ex}.
	Now, consider the \VI-aware \LCP.
	For binary variable $x_1$,
	we compute that
	\begin{align*}
		w_1(0) &= \min \bigl\{x_2 + 0.9x_3 + 0.5x_5 \,\bigm|\, x_2 \ge 1,\ x_3 \ge 1,\ 0.4 \le x_5 \le 0.5, \notag\\
		&\qquad x_2,\ x_3,\ x_4 \in \{0,1\},\ x_5 \in [0,3]\bigr\} = 2.1,\\
		w_1(1) &= \min \bigl\{x_2 + 0.9x_3 + 0.5x_5 \,\bigm|\, x_4 \ge 1,\ x_2,\ x_3,\ x_4 \in \{0,1\},\ x_5 \in [0,3]\bigr\} = 0.
	\end{align*}
	As $w_1(0) > 2$, 
	we can fix $x_1 = 1$.
	Moreover, using the \VI \eqref{e-x5}, we can additionally fix $x_4 = 1$.
	This demonstrates that equipped with \VIs, the \LCP can derive tighter variable bounds.

\end{example}


\subsubsection{The case $r \in \N \backslash \B$.}\label{case-nonbinary}

Next, we attempt to derive the tightest lower and upper bounds of $x_r$ from the system \eqref{linear}--\eqref{intCons} and \eqref{dp-imp-nonb1}--\eqref{dp-imp-nonb4}.
We first consider the case $r \in \N\backslash \I$, i.e., $x_r$ is a continuous variable.
In this case, 
deriving the tightest bounds $\LB \le x_r \le \UB$ from the system \eqref{linear}--\eqref{intCons} and \eqref{dp-imp-nonb1}--\eqref{dp-imp-nonb4} is equivalent to solving the following two optimization problems:
\begin{subequations}
	\begin{align}
		\LB &= \min\left\{x_r ~\mid~ \eqref{linear}\text{--}\eqref{intCons},~\eqref{dp-imp-nonb1}\text{--}\eqref{dp-imp-nonb4} \right\}, \label{nonb-prob1}\\
		\UB &= \max\left\{x_r ~\mid~ \eqref{linear}\text{--}\eqref{intCons},~\eqref{dp-imp-nonb1}\text{--}\eqref{dp-imp-nonb4} \right\}. \label{nonb-prob2}
	\end{align}
\end{subequations}

{Note that for $i \in \B_0^<(r) \cap \B_1^>(r)$ with $u_{ir} < \ell_{ir}$, if $x_r \in (u_{ir}, \ell_{ir})$, then  \eqref{dp-imp-nonb2} implies $x_i = 0$, but \eqref{dp-imp-nonb3} implies $x_i = 1$, a contradiction;
therefore, problems \eqref{nonb-prob1}--\eqref{nonb-prob2} with $x_r \in (u_{ir}, \ell_{ir})$ are infeasible.
Similarly, for $i \in \B_0^>(r) \cap \B_1^<(r)$ with $u_{ir} < \ell_{ir}$,  problems \eqref{nonb-prob1}--\eqref{nonb-prob2} with $x_r \in (u_{ir}, \ell_{ir})$ are also infeasible.
Let
\begin{align}
	&\D := \cup_{i\in
	\left(\B^<_1(r)\cap \B^>_0(r)\right)
	\cup
	\left(\B^<_0(r)\cap \B^>_1(r)\right),~
	u_{ir} < \ell_{ir}} \left(u_{ir}, \ell_{ir}\right), \label{infdefine}
\end{align}
be the set of such infeasible points. 
Similar to the binary case, we can define the implied activity for  $x_r=d  \in \left[0,u_r\right] \backslash \D$ as follows.}

\begin{definition}\label{def-impliedact}
	For linear constraint \eqref{linear}, the implied activity $w_r(d)$  with respect to $r \in \N \backslash \I$ and $d \in \left[0,u_r\right] \backslash \D$ is defined as:
	\begin{equation}\label{tmeq2}
		w_r(d) := \min_{x_r = d} \left\{\sum_{i \in \N} a_i x_i~\Bigm|~\eqref{bounds}\text{--}\eqref{intCons},~ \eqref{dp-imp-nonb1}\text{--}\eqref{dp-imp-nonb4} \right\}.
	\end{equation}
\end{definition}

\noindent
By definition, $w_r(d) \le b_0$ holds if and only if there exists a solution ${x}$ with ${x}_r = d$ satisfying \eqref{linear}--\eqref{intCons} and \eqref{dp-imp-nonb1}--\eqref{dp-imp-nonb4}.
Therefore, we can reformulate \eqref{nonb-prob1} and \eqref{nonb-prob2} as the following two problems:
\begin{subequations}
	\begin{align}
		\label{re-nonb-prob1}
		\LB &= \min\left\{{d \in \left[0, u_r\right] \backslash \D} ~\mid~ w_r(d) \le b_0 \right\}, \\
		\label{re-nonb-prob2}
		\UB &= \max\left\{{d \in \left[0, u_r\right] \backslash \D} ~\mid~ w_r(d) \le b_0 \right\}.
	\end{align}
\end{subequations}

To determine the lower and upper bounds \LB and \UB for variable $x_r$, we need to compute the minimum and maximum of {$d \in \left[0, u_r\right] \backslash \D$} such that $w_r(d) \le b_0$ holds.
We first discuss the computation of $w_r(d)$ for a given {$d \in \left[0, u_r\right] \backslash \D$}, which, as in the  binary case, can be done by analyzing implications \eqref{dp-imp-nonb1}--\eqref{dp-imp-nonb4}.
In fact, from $x_r = d$ and \eqref{dp-imp-nonb1}--\eqref{dp-imp-nonb4}, we obtain $x_i = 1$ for $i \in \B_1^<(r)$ with $d < \ell_{ir}$, 
$x_i = 0$ for $i \in \B_0^<(r)$ with $d < \ell_{ir}$,
$x_i = 1$ for $i \in \B_1^>(r)$ with $d > u_{ir}$, and
$x_i = 0$ for $i \in \B_0^>(r)$ with $d > u_{ir}$.
Note that from (iii) of \cref{remark1}, it must follow that $\ell_{ir} \le u_{ir}$ for all $i \in \B_1^<(r) \cap \B_1^>(r)$.
Therefore, one optimal solution $x^*$ for problem \eqref{tmeq2} is given as
\begin{subequations}
	\setlength{\columnsep}{0.5em}
	\begin{multicols}{2}
		\noindent
		\begin{align*}
			x^*_r &= d, \\
			x^*_i &= 1, ~\forall~i \in \tilde{\B}_1^< (r): = \{i \in \B_1^<(r) \mid d < \ell_{ir}\},
		\end{align*}
		\begin{align*}
			x^*_i &= 1, ~\forall~i \in \tilde{\B}_1^> (r):=\{i \in \B_1^>(r) \mid  d > u_{ir}\},\\
			x^*_i &= 0, ~\forall~i \in \N\backslash (\tilde{\B}_1^< (r) \cup \tilde{\B}_1^> (r) \cup \{r\}),
		\end{align*} 
	\end{multicols}
\end{subequations}\noindent
and the implied activity $w_r(d)$ can be computed as
\begin{equation}\label{wrd}
w_r(d) = \sum_{\substack{i \in \B^{<}_1(r),~ d < \ell_{ir}}} a_i  + \sum_{\substack{i \in \B^{>}_1(r),~ d > u_{ir}}} a_i  + a_r d,~ {\forall~ d \in [0, u_r] \backslash \D}.
\end{equation}

Unfortunately, in contrast to the case $r \in \B$, we cannot simply solve 
\eqref{re-nonb-prob1} and \eqref{re-nonb-prob2} by enumerating $d$, as there may be infinitely many choices for {$d \in \left[0, u_r\right] \backslash \D$}.
To bypass this difficulty, we first note that function \rev{$w_r(\cdot)$} is a piecewise linear function on {$\left[0, u_r\right] \backslash \D$}.
{To see this, let
\begin{equation}
	\{\ell_{ir}\}_{i \in \B_1^<(r)} \cup \{u_{ir}\}_{i \in \B_1^>(r)} := \left\{d_1, \ldots, d_s\right\},
\end{equation}
where $0  =: d_0 < d_1 < \cdots < d_s < d_{s+1} := u_r$ and $s = \left|\{\ell_{ir}\}_{i \in \B_1^<(r)} \cup \{u_{ir}\}_{i \in \B_1^>(r)}\right|$.
Then, the feasible domain can be expressed as $[0, u_r] \backslash \D = \cup_{t=0}^{s} \left([d_t, d_{t+1}] \backslash \D\right)$.
Note that from \eqref{infdefine}, $\D$ is the union of open intervals.
Therefore, for each $t \in \left\{ 0, \ldots, s \right\}$, the set $[d_t, d_{t+1}] \backslash \D$ can be decomposed into the union of (several) closed intervals, and the feasible region $[0, u_r] \backslash \D $ can be decomposed as follows:
\begin{equation}\label{d-define-1}
	[0, u_r] \backslash \D = \cup_{k=1}^{\tau} \left[\dunderbar_k, \bar{d}_k\right],
\end{equation} 
where $\tau \in \mathbb{Z}^+$ is the total number of  closed intervals, 
$\left[\dunderbar_k, \bar{d}_k\right]\subseteq [d_t, d_{t+1}]$ holds for some $t \in \left\{ 0, \ldots, s \right\}$,
and $\dunderbar_k,\bar{d}_{k} \in \{0,u_r\} \cup \left\{\ell_{ir} ~\mid~ i \in \B_1^<(r) \cup \B_0^<(r)\right\} \cup \left\{u_{ir} ~\mid~ i \in \B_1^>(r) \cup \B_0^>(r)\right\}$ satisfying
\begin{equation}\label{d-define-2}
	0 =: \dunderbar_1 < \dunderbar_2  < \cdots < \dunderbar_\tau ,~~\quad \dunderbar_k \leq \bar{d}_k \leq \dunderbar_{k+1}\rev{,}~\forall~k \in \{1,\ldots,\tau-1\},~\quad\bar{d}_\tau=u_r.
\end{equation}}%
{
Then, for $d', d'' \in (\dunderbar_k, \bar{d}_k)$, 
as $\left[\dunderbar_k, \bar{d}_k\right]\subseteq [d_t, d_{t+1}]$ holds for some $t \in \left\{ 0, \ldots, s \right\}$,
it must follow 
$\{ i\in \B_1^<(r)\mid d' < \ell_{ir} \}=\{ i\in \B_1^<(r)\mid d'' < \ell_{ir} \}$ and 
$\{ i\in \B_1^>(r)\mid d' > u_{ir} \}=\{ i\in \B_1^>(r)\mid d'' > u_{ir} \}$.}
Therefore, it follows from \eqref{wrd} that
\begin{lemma}\label{tmplemma1}
	For any {$k \in \{1,\dots,\tau\}$}, if $a_r=0$, then $w_r(\cdot)$ is a constant on {$(\dunderbar_k, \bar{d}_k)$};  
	otherwise, \rev{$w_r(\cdot)$} is a strictly increasing linear function over {$(\dunderbar_k, \bar{d}_k)$}.
\end{lemma}
In online Appendix \ref{appendixC}, we provide an example to illustrate {how the feasible domain can be decomposed into the form of \eqref{d-define-1},
and to demonstrate} that although the function \rev{$w_r(\cdot)$} is piecewise linear, it is {not monotonic on the feasible domain $[0, u_r] \backslash \D$.}

 

To solve the two optimization problems \eqref{re-nonb-prob1} and \eqref{re-nonb-prob2}, 
we can divide them into {$2\tau$} problems:
\begin{subequations}
	\begin{align}
		\LB_k &:= \min\left\{d \in {[\dunderbar_k, \bar{d}_k]} ~\mid~ w_r(d) \le b_0\right\},~\forall~k \in {\{1,\dots, \tau\}}, \label{lbk} \\
		\UB_k &:= \max\left\{d \in {[\dunderbar_k, \bar{d}_k]} ~\mid~ w_r(d) \le b_0\right\},~\forall~k \in {\{1,\dots, \tau\}}, \label{ubk}
	\end{align}
\end{subequations}
and compute the optimums \LB and \UB as:
\begin{equation}\label{minlbk}
	\LB = \min_{k \in {\{1,\ldots, \tau\}}}~ \LB_k,~ \UB = \max_{k \in {\{1,\ldots, \tau\}}}~ \UB_k.
\end{equation}
{Note that for $k \in \left\{1, \ldots,\tau\right\}$ with $\dunderbar_k = \bar{d}_k$, if $w_r(\dunderbar_k) \le b_0$, then
$\LB_k = \UB_k = \dunderbar_k$; otherwise  $\LB_k = +\infty$ and $\UB_k = -\infty$.
Therefore, to compute $\LB_k$ and $\UB_{k}$, it suffices to consider the case with $\dunderbar_k < \bar{d}_k$.
To do this, we need the notation of the left limit of $w_r(d)$ at the right endpoint $\rev{d=}$ $\bar{d}_k$, and the right limit of $w_r(d)$ at the left endpoint $\rev{d=}$ $\dunderbar_k$.}
{\begin{definition}\label{limit}
	For $k \in \{1,\ldots,\tau\}$ with $\dunderbar_{k} < \bar{d}_k$,
	the left limit of $w_r(d)$ at $d = \bar{d}_k$ 
	and the right limit of $w_r(d)$ at $d = \dunderbar_k$ 
	are defined as
	\begin{align}
		w_r^-(\bar{d}_k) := \lim_{\varepsilon \to 0^+} w_r(\bar{d}_k - \varepsilon) &= w_r(\bar{d}_k) + \sum_{i \in \B_1^<(r),~ \bar{d}_k = \ell_{ir}} a_i,\\
	   w_r^+(\dunderbar_k) := \lim_{\varepsilon \to 0^+} w_r(\dunderbar_k + \varepsilon) &= w_r(\dunderbar_k) + \sum_{i \in \B_1^>(r),~ \dunderbar_k = u_{ir}} a_i. \label{def-w+}
   \end{align}
\end{definition}}
\noindent {By definition, it follows 
\begin{lemma}\label{tmplemma2}
	For $k \in \{1,\ldots,\tau\}$ with $\dunderbar_{k} < \bar{d}_k$, $w_r^-(\bar{d}_k) \geq w_{r}(\bar{d}_k)$ and $w_r^+(\dunderbar_k) \geq w_{r}(\dunderbar_k)$ hold.
\end{lemma}}
Using \cref{tmplemma1,tmplemma2}, we can provide closed formulas for $\LB_k$ and $\UB_{k}$.
\begin{proposition}\label{non-01-prop}
	Let $k \in \{1,\dots, \tau\}$, and $\LB_k$ and $\UB_k$ be defined in \eqref{lbk}--\eqref{ubk}.
	Then,
	\begin{itemize}
		\item[(i)] If $w_r(\dunderbar_k) \le b_0$ and $w_r(\bar{d}_k) \le b_0$, then $\LB_k = \dunderbar_k$ and $\UB_k = \bar{d}_k.$
		\item[(ii)] If $w_r(\dunderbar_k) \le b_0$, $w_r(\bar{d}_k) > b_0$, and $w_r^+(\dunderbar_k) > b_0$, then $\LB_k = \UB_k = \dunderbar_k.$
		\item[(iii)] If $w_r(\dunderbar_k) \le b_0$, $w_r(\bar{d}_k) > b_0$, and $w_r^+(\dunderbar_k) \le b_0$, then $a_r > 0$,  $\LB_k = \dunderbar_k$, and $\UB_k = \dunderbar_k + \frac{b_0 - w_r^+(\dunderbar_k)}{a_r}.$
		\item[(iv)] If $w_r(\dunderbar_k) > b_0$ and  $w_r(\bar{d}_k) \le b_0$, then $\LB_k = \UB_k = \bar{d}_k.$
		\item[(v)] If $w_r(\dunderbar_k) > b_0$ and $w_r(\bar{d}_k) > b_0$, then $\LB_k=+\infty$ and $\UB_k=-\infty$.
	\end{itemize}
\end{proposition}
In \cref{example3} of \cref{appendixC}, we provide an illustrative example for applying \cref{non-01-prop} to compute $\{\LB_k\}_{k=1, \ldots, \tau}$ and $\{\UB_k\}_{k=1, \ldots, \tau}$.


Note that for each fixed {$k \in \{1,\dots, \tau\}$}, the computation of $\LB_k$ and $\UB_k$ in \cref{non-01-prop} can be accomplished by computing {$w_r(\dunderbar_k)$, $w_r^+(\dunderbar_k)$, and $w_r(\bar{d}_{k})$ with the 
complexity \rev{of} $\CO(|\B_1^> (r)| + |\B_1^<(r)|)$.}
Therefore, a direct computation of \LB and \UB from \eqref{minlbk} requires
$\CO(\tau (|\B_1^>(r)| + |\B_1^<(r)|))
=\CO\bigl((|\B_1^<(r)| + |\B_0^<(r)| + |\B_1^>(r)| + |\B_0^>(r)|)\allowbreak
(|\B_1^>(r)| + |\B_1^<(r)|)\bigr)$ time.
However, this complexity can be improved to $\CO(|\B_1^<(r)| + |\B_0^<(r)| + |\B_1^>(r)| + |\B_0^>(r)|)$.
To see this, for $k \in \{1,\ldots,\tau-1\}$, we can update $w_r(\dunderbar_{k+1})$ from $w_r(\dunderbar_k)$ as
\begin{equation}
	\label{w-update-1}
	\begin{aligned}
		w_r(\dunderbar_{k+1}) &= \sum_{i \in \B_1^<(r),~ \dunderbar_{k+1} < \ell_{ir}} a_i + \sum_{i \in \B_1^>(r),~ \dunderbar_{k+1} > u_{ir}} a_i + a_r \dunderbar_{k+1} \\ 
		&= w_r(\dunderbar_k) - \sum_{i \in \B_1^<(r),~ \dunderbar_k < \ell_{ir} \le \dunderbar_{k+1}} a_i + \sum_{i \in \B_1^>(r),~ \dunderbar_k \le u_{ir} < \dunderbar_{k+1}} a_i + a_r (\dunderbar_{k+1} - \dunderbar_k).
	\end{aligned}
\end{equation}
Using the fact that the computation of $\left\{\sum\limits_{i \in \B_1^<(r),~ \dunderbar_k < \ell_{ir} \le \dunderbar_{k+1}} a_i\right\}_{k \in \{1,\dots, \tau-1\}}$ and $\left\{ \sum\limits_{i \in \B_1^>(r),~ \dunderbar_k \le u_{ir} < \dunderbar_{k+1}} a_i\right\}_{k \in \{1,\dots, \tau-1\}}$ can be achieved with the complexity of $\CO(|\B_1^<(r)| + |\B_0^<(r)| + |\B_1^>(r)| + |\B_0^>(r)|)$, we can compute $w_r(\dunderbar_k)$ for all $k \in \{1, \ldots, \tau\}$ in the complexity \rev{of} $\CO(|\B_1^<(r)| + |\B_0^<(r)| + |\B_1^>(r)| + |\B_0^>(r)|)$.
Using a similar manner, we can also compute $w_r(\bar{d}_k)$ for all $k \in \{1, \ldots, \tau\}$ in \rev{the} complexity of $\CO(|\B_1^<(r)| + |\B_0^<(r)| + |\B_1^>(r)| + |\B_0^>(r)|)$.
Finally, using the fact that the computation of $\left\{ \sum\limits_{i \in \B_1^>(r),~ \dunderbar_k = u_{ir}} a_i\right\}_{k \in \{1,\dots, \tau\}}$ can be achieved with the complexity of $\CO(|\B_1^<(r)| + |\B_0^<(r)| + |\B_1^>(r)| + |\B_0^>(r)|)$, 
we can also compute $w_r^+(\dunderbar_{k})$ as in \eqref{def-w+} for all $k \in \{1, \ldots, \tau\}$ in the complexity of $\CO(|\B_1^<(r)| + |\B_0^<(r)| + |\B_1^>(r)| + |\B_0^>(r)|)$.

The extension to the case $r\in \I\backslash \B$ (i.e., $x_r$ is an integer variable) is easy.
Indeed, when $r \in \I\backslash \B$, 
similar to \eqref{re-nonb-prob1}--\eqref{re-nonb-prob2}, problems \eqref{nonb-prob1}--\eqref{nonb-prob2}   can be written as
\begin{subequations}
	\begin{align}
		\label{re-nonb-prob1-int}
		\LB &= \min\left\{d \in {\left(\left[0, u_r\right] \backslash \D\right)} \cap \mathbb{Z} ~\mid~ w_r(d) \le b_0 \right\}, \\
		\label{re-nonb-prob2-int}
		\UB &= \max\left\{d \in {\left(\left[0, u_r\right] \backslash \D\right)} \cap \mathbb{Z} ~\mid~ w_r(d) \le b_0 \right\},
	\end{align}
\end{subequations}
where the implied activity $w_r(d)$ {is now defined on $d \in \left(\left[0, u_r\right] \backslash \D\right) \cap \mathbb{Z}$ as
\begin{equation}
	w_r(d)= \sum_{\substack{i \in \B^{<}_1(r),~ d < \ell_{ir}}} a_i  + \sum_{\substack{i \in \B^{>}_1(r),~ d > u_{ir}}} a_i  + a_r d,~ \forall~ d \in \left(\left[0, u_r\right] \backslash \D\right) \cap \mathbb{Z}.
\end{equation}}%
We can compute $\LB$ and $\UB$ using \eqref{minlbk}, where $\LB_k$ and $\UB_k$ for {$k = 1, \ldots, \tau$} can be adapted as  
\begin{subequations}
	\begin{align}
		\LB_k &:= \min\left\{d \in {[\dunderbar_k, \bar{d}_k]} \cap \mathbb{Z} ~\mid~ w_r(d) \le b_0\right\}, \label{lbk-I} \\
		\UB_k &:= \max\left\{d \in {[\dunderbar_k, \bar{d}_k]} \cap \mathbb{Z} ~\mid~ w_r(d) \le b_0\right\}, \label{ubk-I}
	\end{align}
\end{subequations}
and $\{\dunderbar_k, \bar{d}_k\}_{k = 1,\dots, \tau}$ are the same as the continuous case in \eqref{d-define-1}--\eqref{d-define-2}.
Note that \cref{non-01-prop} can be directly extended to compute the $\LB_k$ and $\UB_k$ in \eqref{lbk-I}--\eqref{ubk-I}.
In particular, from \cref{assum1}(ii), and $\dunderbar_k,\bar{d}_{k} \in \{0,u_r\} \cup \left\{\ell_{ir} ~\mid~ i \in \B_1^<(r) \cup \B_0^<(r)\right\} \cup \left\{u_{ir} ~\mid~ i \in \B_1^>(r) \cup \B_0^>(r)\right\}$, \rev{$\{\dunderbar_k, \bar{d}_k\}_{k = 1,\dots, \tau}$ are integral}, and therefore cases (i)--(ii) and (iv)--(v) of \cref{non-01-prop} still hold for the computation of $\LB_k$ and $\UB_k$  in \eqref{lbk-I}--\eqref{ubk-I}.
For case (iii), the lower bound is still {$\LB_k = \dunderbar_k$}, and as $r \in \I \backslash \B$, the upper bound $\UB_k$ can be adapted as {$\UB_k = \dunderbar_k +  \lfloor \frac{b_0 - w_r^+(\dunderbar_k)}{a_r} \rfloor$}.

\begin{remark}\label{remark3}
	 Similar to the binary case in \cref{remark2}, if the bounds of $x_r$ are tightened, we may additionally derive tighter bounds  for other variables using \VIs in \eqref{dp-imp-nonb1}--\eqref{dp-imp-nonb4}.
	Specifically, if $x_r \ge \LB$, then from \eqref{dp-imp-nonb3} and \eqref{dp-imp-nonb4}, we obtain $x_i = 1$ for $i \in \B_1^>(r)$ with $\LB > u_{ir}$ and $x_i = 0$ for $i \in \B_0^>(r)$ with $\LB > u_{ir}$;
	if $x_r \le \UB$, then from \eqref{dp-imp-nonb1} and \eqref{dp-imp-nonb2}, we obtain $x_i = 1$ for $i \in \B_1^<(r)$ with $\UB < \ell_{ir}$ and $x_i = 0$ for $i \in \B_0^<(r)$ with $\UB < \ell_{ir}$. 
\end{remark}

\subsection{Comparison with the presolve technique in \citet{Achterberg2013a}}\label{compare}

\citet{Achterberg2013a} proposed a presolve technique  to tighten variable bounds, achieved by applying probing on the binary variables of the cover implied by a knapsack constraint and observing that at least one of them must be set to zero.
The computational results of \citet{Achterberg2013a} showed that this technique could affect a substantial number of \MIP problems \rev{in their testbed} and improve the performance of the \MIP solver CPLEX.
Similar ideas have also been explored in the stochastic programming community, see e.g., \cite{Lejeune2010,Luedtke2010},
and recently, this technique was also implemented in the SAS optimization solvers \citep{Pratt2025}.
In this subsection, we will compare the proposed \VI-aware \LCP with the presolve technique in \cite{Achterberg2013a}.
In particular, our analysis shows that our proposed approach could derive tighter bounds than those derived by the presolve technique in \cite{Achterberg2013a}.
{Throughout this subsection, for simplicity, we only discuss the tightening of bounds for continuous variables; the discussions for binary and general integer variables are similar.}

We first review the presolve technique in \cite{Achterberg2013a}. 
Consider 
a knapsack constraint of the form:
\begin{align}
	& a_{\B}^\top x_{\B} \leq b_0, \label{knap}
\end{align}
where $a_i \geq 0$ for all $i \in \B$.
A set $\V \subseteq \B$ is a cover implied by \eqref{knap} if $\sum_{i \in \V} a_i > b_0$. 
Using the cover information and the \VIs in \eqref{dp-imp-nonb1}, we may tighten the lower bound of variable $x_r$. 
Specifically, for variable $x_r$ with $r \in \N\backslash \I$, if $\V \subseteq \B_1^<(r)$, then $x_i =0 \rightarrow  x_r \geq \ell_{ir}$ (implied by \eqref{dp-imp-nonb1}) holds for all $i \in \V$.
Since at least one variable in cover $\V$ should be zero (implied by the knapsack constraint \eqref{knap}), we can tighten the lower bound of $x_r$ as $x_r \geq \min_{i \in \V} \ell_{ir}$.
Note that different covers may yield different lower bounds of variable $x_r$.
To obtain the tightest lower bound $\LBhat$ for variable $x_r$, \citet{Achterberg2013a} proposed to find the ``best'' cover $\V$ by adding indices $i \in \B_1^<(r)$ in a non-increasing order of $\ell_{ir}$ into $\V$ until it defines a cover of the knapsack constraint \eqref{knap}.
This is equivalent  to solving the following optimization problem:
\begin{align}\label{2013bound-1}
	\LBhat = \max \left\{\ell_{i_0r},~ i_0 \in \B_1^<(r) ~\Bigm|~ \sum_{i \in \B_1^<(r),~ \ell_{i_0r} \le \ell_{ir}} a_i > b_0\right\}.
\end{align}


\noindent Similarly, using the cover information and the \VIs in \eqref{dp-imp-nonb3}, we may tighten the upper bound of variable $x_r$. 
Specifically, if $\V \subseteq \B_1^>(r)$, then $x_i =0 \rightarrow  x_r \leq u_{ir}$ (implied by \eqref{dp-imp-nonb3}) holds for all $i \in \V$, and we can tighten the upper bound of $x_r$ as $x_r \le \max_{i \in \V} u_{ir}$.
To obtain the tightest upper bound $\UBhat$ for variable $x_r$, we can find the ``best'' cover $\V$ by adding indices $i \in \B_1^>(r)$ in a non-decreasing order of $u_{ir}$ into $\V$ until it defines a cover of the knapsack constraint \eqref{knap}.
This is equivalent to solving the following optimization problem:
\begin{align}\label{2013bound-2}
	\UBhat = \min \left\{u_{i_0r},~ i_0 \in \B_1^>(r) ~\Bigm|~ \sum_{i \in \B_1^>(r),~ u_{i_0r} \ge u_{ir}} a_i > b_0\right\}.
\end{align}

Next, we compare the proposed \VI-aware \LCP in \cref{subsect:boundtightening-theory} with the presolve technique of \citet{Achterberg2013a} on knapsack constraints.
We first note that from \eqref{re-nonb-prob1}--\eqref{re-nonb-prob2} and \eqref{wrd}, the lower bound $\LB$ and upper bound $\UB$  for variable $x_r,~ r \in \N \backslash \I$, detected by the proposed \VI-aware \LCP, can be written as
\begin{subequations}
	\begin{align}
		\label{prob1}
		\LB = \min \left\{{d \in \left[0, u_r\right] \backslash \D} ~\Bigm|~ \sum_{\substack{i \in \B^{<}_1(r),~ d < \ell_{ir}}} a_i  + \sum_{\substack{i \in \B^{>}_1(r),~ d > u_{ir}}} a_i \le b_0\right\}, \\
		\label{prob2}
		\UB = \max \left\{{d \in \left[0, u_r\right] \backslash \D} ~\Bigm|~ \sum_{\substack{i \in \B^{<}_1(r),~ d < \ell_{ir}}} a_i  + \sum_{\substack{i \in \B^{>}_1(r),~ d > u_{ir}}} a_i \le b_0\right\}.
	\end{align}
\end{subequations}
Comparing \eqref{2013bound-1}--\eqref{2013bound-2} and \eqref{prob1}--\eqref{prob2}, 
we observe that the presolve technique in \cite{Achterberg2013a} uses only \VIs \rev{~in} \eqref{dp-imp-nonb1} to derive the lower bound $\LBhat$, and only \VIs \rev{~in} \eqref{dp-imp-nonb3} to compute the upper bound $\UBhat$, while the proposed \VI-aware \LCP simultaneously uses both the \VIs \rev{~in} \eqref{dp-imp-nonb1} and \eqref{dp-imp-nonb3} to derive the lower and upper bounds \LB and \UB for variable $x_r$, and additionally excludes the infeasible domain $\D$ using \VIs \rev{~in} \eqref{dp-imp-nonb1}--\eqref{dp-imp-nonb4}.
The following theorem further demonstrates this advantage.

\begin{theorem}\label{2013theorem}
	Let $\LBhat$ and $\UBhat$ be defined in \eqref{2013bound-1}--\eqref{2013bound-2}, and $\LB$ and $\UB$ be defined in \eqref{prob1}--\eqref{prob2}.
	Then, it must follow $\LB \ge \LBhat$ and $\UB \le \UBhat$.
\end{theorem}
\begin{proof}{Proof}
	\raggedright
	We first prove $\LB \ge \LBhat$.
	If $\LBhat = 0$, $\LB \ge \LBhat$ holds trivially;
	otherwise, $\LBhat > 0$.
	For any {$d \in \left[0, \LBhat\right) \backslash \D$}, it follows 
	\begin{equation}\label{lbproof}
		\sum_{\substack{i \in \B^{<}_1(r),~ d < \ell_{ir}}} a_i  + \sum_{\substack{i \in \B^{>}_1(r),~ d > u_{ir}}} a_i \geq \sum_{\substack{i \in \B^{<}_1(r),~ d < \ell_{ir}}} a_i
		\ge \sum_{i \in \B_1^< (r) ,~\LBhat \le \ell_{ir}} a_i \stackrel{(a)}{>} b_0,
	\end{equation}
	where (a) follows from the definition of $\LBhat$ in \eqref{2013bound-1}.
	Thus, {$d \in \left[0, \LBhat\right) \backslash \D$} cannot be a feasible solution of \eqref{prob1}, 
	implying $\LB \ge \LBhat$.
	Similarly, 	if $\UBhat = u_r$, $\UB \le \UBhat$ holds trivially;
	otherwise,  $\UBhat < u_r$.
	Then, for any {$d \in \left(\UBhat, u_r\right] \backslash \D$}, it follows
	\begin{equation}\label{ubproof}
		\sum_{\substack{i \in \B^{<}_1(r),~ d < \ell_{ir}}} a_i  + \sum_{\substack{i \in \B^{>}_1(r),~ d > u_{ir}}} a_i \ge \sum_{\substack{i \in \B^{>}_1(r),~ d > u_{ir}}} a_i \ge \sum_{i \in \B_1^> (r),~ \UBhat \ge u_{ir}} a_i \stackrel{(b)}{>} b_0,
	\end{equation}
	where (b) follows from the definition of $\UBhat$ in \eqref{2013bound-2}.
	Thus, {$d \in \left(\UBhat, u_r\right] \backslash \D$} cannot be a feasible solution of \eqref{prob2}, 
	implying $\UB \le \UBhat$.
\end{proof}


\cref{2013theorem} shows that when applied to knapsack constraints, the proposed \VI-aware \LCP can derive variable bounds that are at least as tight as those derived from the approach of \cite{Achterberg2013a}.
In \cref{tmpex} given in online Appendix \ref{appendixA}, we further show that our approach could derive strictly tighter variable bounds.

Finally, for an arbitrary linear constraint that involves non-binary variables with non-zero coefficients, the presolve technique of \citet{Achterberg2013a} first \rev{relaxes} non-binary variables to obtain a knapsack constraint using variable bounds \eqref{bounds} or \rev{\VIs in} \eqref{impc1}--\eqref{impc4}, and then \rev{derives} tighter variable bounds $\LBhat$ and $\UBhat$ based on the \rev{resulting} knapsack constraint. 
In contrast, our proposed \VI-aware \LCP can be directly applied to an arbitrary linear constraint (and does not require relaxing non-binary variables from the constraint), thereby retaining more information for deriving variable bounds.
In \cref{tmp-eex} given in online Appendix \ref{appendixA}, we show that this could enable us to derive strictly tighter variable bounds.



\subsection{Algorithm framework}\label{subsect:algorithm}

For a linear constraint \eqref{linear}, we can apply the \VI-aware \LCP  to tighten the lower and upper bounds for each variable independently.
However, this could be time-consuming.
Here, we extend the approach of \citet{Achterberg2013a}, that tightens all variable bounds simultaneously based on knapsack constraints, onto our setting. 
The details are summarized in \cref{alg:boundtightening-binary} in online Appendix \ref{appendixB1}.

In steps \ref{line-bin-init}--\ref{line-bin-fix-end} of \cref{alg:boundtightening-binary}, we first  tighten the lower and upper bounds of binary variables.
Specifically,  in step \ref{line-bin-init}, we initialize $w_r(0) := 0$ and $w_r(1) := a_r$ for all $r \in \B$.
Here, we mark that $a_r$ is included in the computation of $w_r(1)$ to avoid overcounting the coefficients.
In steps \ref{line-nonbin-loop-start}--\ref{line-nonbin-loop-end}, we iterate over all \VIs that contain a non-binary variable $x_j$ with $a_j > 0$ in the \VI graph $\G$, to account for the contribution of non-binary variables.
In steps \ref{line-bin-loop-start}--\ref{line-bin-loop-end}, we iterate over all cliques that contain a binary variable $x_i$ with $a_i > 0$, to account for the contribution of binary variables.
In steps \ref{line-bin-fix-start}--\ref{line-bin-fix-end}, we attempt to fix binary variables using \cref{0-1prop}.
Subsequently, in steps \ref{line-nbin-init}--\ref{line-nbin-end} of \cref{alg:boundtightening-binary}, we attempt to tighten the lower and upper bounds of non-binary variables.
In particular, in steps \ref{line-bin-imp-1-start}--\ref{line-bin-imp-2-end}, we compute the implied activities $w_r(0)$ and $w_r(u_r)$ for all $r \in \N \backslash \B$.
In steps \ref{line-nonb-lbub-start}--\ref{line-nonb-lbub-end}, \rev{we solve problems \eqref{lbk}--\eqref{ubk} using \cref{non-01-prop}, and compute the lower and upper bounds for all non-binary variables using \eqref{minlbk}.}%

	In computational experience, we observed that querying the clique table in steps \ref{line-bin-loop-start}--\ref{line-bin-loop-end} of  \cref{alg:boundtightening-binary} could be time-consuming,  especially for problems with many large cliques.
	To further improve the performance of \cref{alg:boundtightening-binary}, we present two techniques extended from \citet{Achterberg2013a}, to avoid unnecessary computations.
	The details of these two techniques are discussed in online Appendix \ref{appendixB2}.

\subsection{Algorithm enhancements}
\label{extent}

In this section, we derive two enhancements of our \VI-aware \LCP.

\textit{Utilizing the objective function.} 
In \cref{subsect:boundtightening-theory}, the proposed \VI-aware \LCP derives variable bounds based on a linear constraint \eqref{linear} from problem \eqref{MIP}.
The derived bounds hold for all feasible solutions of problem \eqref{MIP}.
It is possible to apply the proposed approach based on a linear inequality derived from the objective function along with the bound on the optimal value of the problem.
Specifically, letting $\nu_\LB$ be the lower bound of the optimal objective value of \eqref{MIP}, then $-c^\top x \leq -\nu_{\LB}$ \rev{holds} for all optimal solutions $x^*$ of \eqref{MIP}.
Therefore, we can apply the proposed \VI-aware \LCP to derive variable bounds based on $-c^\top x \leq -\nu_{\LB}$.
The derived bounds hold for all optimal solutions of problem \eqref{MIP}.
Similarly, letting $\nu_{\UB}$ be the upper bound of the optimal objective value of \eqref{MIP}, then we can apply the proposed \VI-aware \LCP to derive variable bounds based on $c^\top x \leq \nu_{\UB}$.



\textit{Deriving more \VIs.}  
Apart from deriving global variable bound tightening for \eqref{MIP}, the proposed \VI-aware \LCP can also be used to derive additional \VIs.
As an example, let us consider two binary variables $x_r$ and $x_p$ that cannot be fixed by \cref{alg:boundtightening-binary} and 
the following optimization problem:
\begin{equation}\label{two-1}
	\UB^\prime := \max \left\{x_r + x_p ~\mid~ \eqref{linear}\text{--}\eqref{intCons},~\eqref{dp-imp-binary1}\text{--}\eqref{dp-imp-binary4},~ \eqref{tmp-imp-1}\text{--}\eqref{tmp-imp-4} \right\},
\end{equation}
where \eqref{tmp-imp-1}\text{--}\eqref{tmp-imp-4} are \VIs of $x_p$:
\begin{subequations}
	\begin{align}
		x_p = 0 &\to x_j \ge \ell_{pj}, \quad \forall~j \in \N_0^\ge(p), \label{tmp-imp-1} \\
		x_p = 1 &\to x_j \ge \ell_{pj}, \quad \forall~j \in \N_1^\ge(p), \label{tmp-imp-2}
	\end{align}
	\begin{align}
		x_p = 0 &\to x_j \le u_{pj}, \quad \forall~j \in \N^\le_0(p), \label{tmp-imp-3} \\
		x_p = 1 &\to x_j \le u_{pj}, \quad \forall~j \in \N^\le_1(p), \label{tmp-imp-4}
	\end{align}
\end{subequations}
Compared with \eqref{bin-prob2}, problem \eqref{two-1} incorporates \VIs in \eqref{tmp-imp-1}\text{--}\eqref{tmp-imp-4} of $x_p$, with the objective function $x_r + x_p$.
Now, if $\UB^\prime \le 1$, we can derive a new \VI $x_r + x_p \le 1$.
To check whether $\UB^\prime \le 1$, we define the following implied activity for variables $x_r$ and $x_p$ as follows:
\begin{equation}\label{two-implied}
	w_{r,p}(v_1,v_2) = \min_{x_r = v_1,\ x_p = v_2} \left\{\sum_{j \in \N} a_j x_j ~\Bigm|~  \eqref{bounds}\text{--}\eqref{intCons},~\eqref{dp-imp-binary1}\text{--}\eqref{dp-imp-binary4},~  \eqref{tmp-imp-1}\text{--}\eqref{tmp-imp-4} \right\}.
\end{equation}
Then $w_{r,p}(1,1) > b_0$ if and only if $\UB^\prime \le 1$.
To compute $w_{r,p}(1,1)$, we can fix \rev{$x_r = x_p = 1 $}, and use the \VIs \rev{~in} \eqref{dp-imp-binary1}--\eqref{dp-imp-binary4} and \eqref{tmp-imp-1}--\eqref{tmp-imp-4} to tighten the variable bounds:
\begin{subequations}
	\setlength{\columnsep}{0.5em}
	\begin{multicols}{2}
		\noindent
		\begin{align*}
			x_j& \ge \ell_{rj},~ \forall~ j \in \N_1^\ge(r) \backslash \N_1^\ge(p), \\   
			x_j& \ge \ell_{pj},~ \forall~ j \in \N_1^\ge(p) \backslash \N_1^\ge(r), \\   
			x_j& \ge \max \{\ell_{rj}, \ell_{pj}\},~ \forall~ j \in \N_1^\ge(r) \cap \N_1^\ge(p), 
		\end{align*}
		\begin{align*}
			x_j& \le u_{rj},~ \forall~ j \in \N_1^\le(r) \backslash \N_1^\le(p), \\   
			x_j& \le u_{pj},~ \forall~ j \in \N_1^\le(p) \backslash \N_1^\le(r), \\   
			x_j& \le \min \{u_{rj}, u_{pj}\},~ \forall~ j \in \N_1^\le(r) \cap \N_1^\le(p). 
		\end{align*} 
	\end{multicols}
\end{subequations}\noindent
Note that if there does not exist a point $x$ satisfying the bounds,  then problem \eqref{two-implied} is infeasible and $w_{r,p}(1,1) = +\infty$, implying that $x_r+x_p\leq1$ holds.
Otherwise, 
\begin{equation*}
	w_{r,p}(1,1) = \sum_{j \in \N_1^\ge(r) \backslash \N_1^\ge(p)} a_j \ell_{rj} + \sum_{j \in \N_1^\ge(p) \backslash \N_1^\ge(r)} a_j \ell_{pj} + \sum_{j \in \N_1^\ge(r) \cap \N_1^\ge(p)} a_j \max \{\ell_{rj}, \ell_{pj}\} + a_r + a_p.
\end{equation*}
Similarly, we can derive $x_p \le x_r,~ x_r \le x_p$, and $x_r + x_p \ge 1$, by checking whether $w_{r,p}(0,1)> b_0$, $ w_{r,p}(1,0)>b_0$, and $w_{r,p}(0,0)>b_0$, hold respectively.

Note that deriving \VIs for all possible combinations $\{\left(r, p, v_1, v_2\right) ~\mid~ r,p \in \B,~ v_1 ,v_2 \in \left\{0,1\right\}\}$ could be time-consuming.
To save the computational effort, 
we can skip combinations $\left(r, p, v_1, v_2\right)$ with $w_r(v_1) + w_p(v_2) \le b_0$, as $w_{r,p}(v_1,v_2) \le w_r(v_1) + w_p(v_2) $, where $w_r(v_1)$ and $w_p(v_2)$ are computed by steps \ref{line-bin-loop-start}--\ref{line-bin-loop-end} of \cref{alg:boundtightening-binary} (see the online Appendix \ref{appendixB1}).
In particular, 
we sort $\{ w_r(0) \}_{r \in \B}$ and $\{ w_r(1) \}_{r \in \B}$ in  a non-increasing order  and then, for each $(v_1, v_2) \in \{0,1\}^2$, we apply a two-level loop to iteratively check whether $w_{r,p}(v_1,v_2) > b_0$ if $w_r(v_1) + w_p(v_2) > b_0$.
To further avoid too much computational effort, we stop each loop for 10 consecutive combinations, if no new \VI is derived.


%% file: section_computational.tex
\section{Computational results}\label{section-computational}

In this section, we evaluate the performance impact of the \VI aggregation in \cref{section-aggregation} and the \VI-aware \LCP in \cref{section-boundtightening}.
The two proposed presolve techniques are implemented within the open-source MIP solver HiGHS \citep{Huangfu2018}, version 1.12.0. 
Note that standard clique merging and LCP have been implemented in HiGHS, and will be included in all our computational experiments.
The experiments are executed on a Linux cluster equipped with Intel(R) Xeon(R) Platinum 8358 CPUs running at 2.60 GHz. We employ the serial version of HiGHS with a time limit of 7200 seconds for each run.
We test our algorithms on the MIPLIB~2017 benchmark testset \citep{Gleixner2021}. 
For each of the 240 problems, we use five random seeds, and each problem and seed combination is regarded as an independent observation, referred to as an ``instance''. 
This results in a testbed of 1200 instances used throughout this section.

\subsection{Variable implication aggregation}

\begin{table}[t]
	\renewcommand{\arraystretch}{1.4}
	\addtolength{\tabcolsep}{2pt}
	\centering
	\caption{Performance impact of applying \VI aggregation.}
	\scriptsize
	\begin{tabular}{{|l|r|r|r|r|r|r|r|r|r|r|}} \hline
		&  & \multicolumn{3}{c|}{\Default} & \multicolumn{4}{c|}{\vbca} & \multicolumn{2}{c|}{Compare} \\ \hline
		Bracket & \tblIns & \tblS & \tblT & \tblN & \tblS & \tblT & \tblN & \tblTa & \tblT & \tblN \\ \hline
		$\ge$    0 & 761 &   746 &     242.62 &     798.78 &   755 &     238.19 &     795.22 &       0.18 &       0.98 &       1.00 \\ \hline
		$\ge$   10 & 693 &   678 &     361.28 &    1293.54 &   687 &     352.39 &    1284.53 &       0.20 &       0.98 &       0.99 \\ \hline
		$\ge$  100 & 494 &   479 &     919.39 &    4865.24 &   488 &     886.96 &    4830.84 &       0.25 &       0.96 &       0.99 \\ \hline
		$\ge$ 1000 & 254 &   239 &    2474.22 &   16692.54 &   248 &    2305.15 &   16507.48 &       0.29 &       0.93 &       0.99 \\ \hline
		Affected & 80 &    65 &     961.95 &    2791.34 &    74 &     802.40 &    2639.06 &       0.31 &       0.83 &       0.95 \\ \hline
	\end{tabular}
	\label{table:vbca-MIPLIB2017}
\end{table}

First, we evaluate the performance impact of the \VI aggregation; see \cref{section-aggregation}.
To this end, we compare the default setting of HiGHS, denoted as \Default, with the setting \vbca, in which the proposed \VI aggregation is implemented. 
The results are summarized in \cref{table:vbca-MIPLIB2017}.
In \cref{table:vbca-MIPLIB2017}, the ``$\ge n$'' bracket represents the set of instances that are solved by at least one of the two settings under comparison, for which the slower of the two requires at least $n$ seconds. Increasing $n$ excludes easier instances and thus yields a hierarchy of instance subsets with increasing difficulty \citep{Achterberg2013b}. 
The last row ``Affected'' contains a subset of instances in ``$\ge 0$'' with different solving paths between the two settings, as indicated by changes in either the number of nodes or the simplex iterations \citep{Achterberg2013b}. 
Column \tblIns denotes the number of instances solved by at least one setting (in total, 761 instances are solved by either \Default or \vbca). 
Column \tblS reports the number of instances solved by each setting within each bracket. 
Columns \tblT, \tblN, and \tblTa denote the shifted geometric means of total CPU time (in seconds), node number, and the CPU time used by executing the proposed \VI aggregation (in seconds), respectively, all using a shift of 1. 
Under ``Compare'', we report the ratios of the shifted geometric means for \tblT and \tblN; a value less than 1.0 indicates that \vbca outperforms \Default. 
Detailed statistics of instance-wise computational results in \cref{table:vbca-MIPLIB2017} and the subsequent tables in this paper can be found in the {GitHub repository} \url{https://github.com/Changlong-Lii/VI-Presolve}.

From \cref{table:vbca-MIPLIB2017}, we first observe that the overhead of \VI aggregation is very small; 
it uses 0.18 seconds on average for each instance.
In contrast, applying this technique can achieve a better performance.
Overall, the CPU time returned by \vbca is 2\% smaller than that returned by \Default, with nine more instances being solved.
Moreover, for the 80 affected instances, we can observe a larger improvement, with a 17\% reduction on CPU time and a 5\% reduction on node number.
Note that it is reasonable to see the time reduction exceeds the node reduction, since one primary benefit of the proposed \VI aggregation comes from the reduction on the number of constraints; see \cref{section-aggregation}.
Indeed, for some affected instances, we can observe a tremendous reduction on the number of constraints---for problem \texttt{neos-4763324-toguru} with five seeds, the number of constraints decreases from 106723 to 53824;
for problem \texttt{neos-631710} with five seeds, the number of constraints decreases from 169576 to 4741.   
These results suggest that incorporating the \VI aggregation into HiGHS can indeed make a better performance.

\subsection{Variable implication-aware linear constraint propagation}

\begin{table}[t]
	\renewcommand{\arraystretch}{1.4}
	\addtolength{\tabcolsep}{2pt}
	\centering
	\caption{Performance impact of applying \VI-aware \LCP (with \Default as baseline).}
	\scriptsize
	\begin{tabular}{{|l|r|r|r|r|r|r|r|r|r|r|}} \hline
		&  & \multicolumn{3}{c|}{\Default} & \multicolumn{4}{c|}{\tobj} & \multicolumn{2}{c|}{Compare} \\ \hline
		Bracket & \tblIns & \tblS & \tblT & \tblN & \tblS & \tblT & \tblN & \tblTb & \tblT & \tblN \\ \hline
		$\ge$    0 & 759 &   746 &     240.45 &     806.24 &   751 &     234.01 &     762.64 &       0.85 &       0.97 &       0.95 \\ \hline
		$\ge$   10 & 695 &   682 &     350.32 &    1257.55 &   687 &     340.10 &    1189.56 &       0.95 &       0.97 &       0.95 \\ \hline
		$\ge$  100 & 500 &   487 &     875.40 &    4540.57 &   492 &     826.48 &    4246.95 &       1.16 &       0.94 &       0.94 \\ \hline
		$\ge$ 1000 & 265 &   252 &    2283.61 &   14627.37 &   257 &    2047.67 &   13158.75 &       1.46 &       0.90 &       0.90 \\ \hline
		Affected & 400 &   387 &     378.40 &     470.64 &   392 &     351.68 &     422.31 &       1.64 &       0.93 &       0.90 \\ \hline
	\end{tabular}
	\label{table:3-MIPLIB2017}
\end{table}

In this section, we evaluate the performance impact of the proposed \VI-aware \LCP; see \cref{section-boundtightening}. 
To do this, we first compare \tobj, in which \VI-aware \LCP is applied, with the default setting of \HiGHS (\Default).
The computational results are summarized in \cref{table:3-MIPLIB2017}, where under column \tblTb, we additionally report 
the shifted geometric mean of the CPU time used by executing the proposed \VI-aware \LCP (in seconds, with a shift of 1).
From the table, we observe that the average CPU time spent on implementing  the proposed \VI-aware \LCP is 0.85 seconds, demonstrating that only a modest average overhead is incurred.
In sharp contrast, applying the \VI-aware \LCP can lead to a notable performance improvement compared to the default setting of HiGHS.
Overall, \tobj achieves a reduction of 3\% in CPU time and 5\% in node number, with five more instances being solved.
For the challenging instances that require at least 1000 seconds to be solved by at least one setting, we can even observe a 10\% CPU time reduction and a 10\% node number reduction, suggesting that the performance of \tobj becomes more pronounced as the difficulty of instances increases.
Note that different from the \VI aggregation where only a small fraction of instances is affected, the proposed \VI-aware \LCP affects $400$ instances among $759$ instances, demonstrating its broad applicability within the MIPLIB 2017 benchmark testset.
For these affected instances, \tobj achieves a relatively larger performance improvement; overall, a  7\% CPU time reduction and a 10\% node number reduction can be observed.

\begin{table}[t]
	\renewcommand{\arraystretch}{1.4}
	\addtolength{\tabcolsep}{2pt}
	\centering
	\caption{Comparison of \VI-aware \LCP with the state-of-the-art approach of \settingAchterberg.}
	\scriptsize
	\begin{tabular}{{|l|r|r|r|r|r|r|r|r|r|}} \hline
		&  & \multicolumn{3}{c|}{\settingAchterberg} & \multicolumn{3}{c|}{\tobj} & \multicolumn{2}{c|}{Compare} \\ \hline
		Bracket &           \tblIns &              \tblS &               \tblT &               \tblN &              \tblS &               \tblT &               \tblN &               \tblT &               \tblN \\ \hline
		$\ge$    0   &      758       &       746      &       237.13   &       829.93   &       751      &       232.95   &       776.37   &         0.98   &         0.94   \\ \hline
		$\ge$   10   &      694       &       682      &       344.96   &      1297.44   &       687      &       338.60   &      1212.22   &         0.98   &         0.93   \\ \hline
		$\ge$  100   &      498       &       486      &       855.46   &      4787.79   &       491      &       826.48   &      4419.92   &         0.97   &         0.92   \\ \hline
		$\ge$ 1000   &      260       &       248      &      2238.84   &     16724.29   &       253      &      2097.98   &     14619.18   &         0.94   &         0.87   \\ \hline
		Affected &      388       &       376      &       372.15   &       555.13   &       381      &       354.07   &       485.66   &         0.95   &         0.87   \\ \hline
	\end{tabular}
	\label{table:2013-MIPLIB2017}
\end{table}

Next, we compare the performance of the proposed \VI-aware \LCP with the state-of-the-art approach of \citet{Achterberg2013a}, where we recall that the latter tightens variable bounds based on only knapsack constraints (obtained from the problem or constraint relaxation) and \VIs; see \cref{compare} for a detailed discussion.
The comparison results of \tobj with the presolve technique of \citet{Achterberg2013a} are summarized in \cref{table:2013-MIPLIB2017}.
From \cref{table:2013-MIPLIB2017}, we observe that the proposed \VI-aware \LCP can indeed lead to a better performance than the state-of-the-art approach of \citet{Achterberg2013a}.
Overall, a reduction of 2\% on CPU time and 6\% on node number can be achieved, with five more instances being solved.
For the 388 affected instances, the improvements are more pronounced, with a reduction of 5\% on CPU time and 13\% on node number.

The improvements of the proposed \tobj over the approach of \citet{Achterberg2013a} may come from the following two aspects: (1) better bound tightenings could be derived by the proposed \VI-aware \LCP, as it is based on a general linear constraint (rather than a knapsack relaxation) with more \VI information exploited; see \cref{compare}; and (2) more bound tightenings and \VIs could be derived, as it is enhanced by the two techniques in \cref{extent}.
To further justify this, we perform another experiment in which the \VI-aware \LCP is applied only to linear constraints in the problem without the two enhancements in \cref{extent} (denoted as \scons), and compare it with \tobj and the approach of \citet{Achterberg2013a}.
The results are summarized in \cref{table:2013-scons,table:extension} and are detailed as follows.

\begin{table}[t]
	\renewcommand{\arraystretch}{1.4}
	\addtolength{\tabcolsep}{2pt}
	\centering
	\scriptsize
	\caption{\centering Comparison of \VI-aware \LCP (excluding two enhancements) with the state-of-the-art approach of \settingAchterberg.}
	\begin{tabular}{{|l|r|r|r|r|r|r|r|r|r|}} \hline
		&  & \multicolumn{3}{c|}{\settingAchterberg} & \multicolumn{3}{c|}{\scons} & \multicolumn{2}{c|}{Compare} \\ \hline
		Bracket & \tblIns & \tblS & \tblT & \tblN & \tblS & \tblT & \tblN & \tblT & \tblN \\ \hline
		$\ge$    0 & 749 &   746 &     227.57 &     845.39 &   747 &     227.91 &     840.97 &       1.00 &       0.99 \\ \hline
		$\ge$   10 & 684 &   681 &     333.27 &    1322.88 &   682 &     333.76 &    1315.31 &       1.00 &       0.99 \\ \hline
		$\ge$  100 & 483 &   480 &     849.21 &    5154.30 &   481 &     844.11 &    5085.56 &       0.99 &       0.99 \\ \hline
		$\ge$ 1000 & 240 &   237 &    2281.29 &   18571.06 &   238 &    2250.34 &   18078.11 &       0.99 &       0.97 \\ \hline
		Affected & 62 &    59 &     555.11 &     663.90 &    60 &     524.10 &     619.99 &       0.94 &       0.93 \\ \hline
	\end{tabular}
	\label{table:2013-scons}
\end{table}

First, from \cref{table:2013-scons}, we observe that overall,
applying the proposed \VI-aware \LCP without the two enhancements only slightly outperforms the approach of \citet{Achterberg2013a}. 
For the 62 affected instances, however, the proposed approach can notably improve the performance, obtaining a CPU time reduction of 6\% and a node reduction of 7\%.
This further confirms the results in \cref{compare}, i.e., the proposed \VI-aware \LCP (without the two \rev{enhancements} in \cref{extent}) could detect better bound tightenings than those detected by the approach of \cite{Achterberg2013a}.


\begin{table}[t]
	\renewcommand{\arraystretch}{1.4}
	\addtolength{\tabcolsep}{2pt}
	\centering
	\caption{Performance impact of applying the two enhancements of \VI-aware \LCP.}
	\scriptsize
	\begin{tabular}{{|l|r|r|r|r|r|r|r|r|r|}} \hline
		&  & \multicolumn{3}{c|}{\scons} & \multicolumn{3}{c|}{\tobj} & \multicolumn{2}{c|}{Compare} \\ \hline
		Bracket &         \tblIns &           \tblS &           \tblT &           \tblN &           \tblS &           \tblT &           \tblN &           \tblT &           \tblN \\ \hline
		$\ge$    0   &      757       &       747      &       236.41   &       827.41   &       751      &       231.89   &       777.16   &         0.98   &         0.94   \\ \hline
		$\ge$   10   &      693       &       683      &       343.95   &      1291.43   &       687      &       337.11   &      1211.97   &         0.98   &         0.94   \\ \hline
		$\ge$  100   &      498       &       488      &       843.56   &      4654.43   &       492      &       819.40   &      4347.30   &         0.97   &         0.93   \\ \hline
		$\ge $ 1000   &      256       &       246      &      2221.15   &     17173.08   &       250      &      2126.03   &     15994.35   &         0.96   &         0.93   \\ \hline
		Affected &      380       &       370      &       355.79   &       553.41   &       374      &       340.33   &       487.10   &         0.96   &         0.88   \\ \hline
	\end{tabular}
	\label{table:extension}
\end{table}

To verify the second reason (i.e., the effect of the two enhancements in \cref{extent}), we compare the setting \tobj with \scons.
From \cref{table:extension}
we observe that applying the two enhancements in \VI-aware \LCP can lead to a better performance.
Overall, \tobj achieves a reduction of 2\% in CPU time and 6\% in node number.
For the 380 affected instances, we can even observe a CPU time reduction of 4\% and a node reduction of 12\%.
This demonstrates that the two proposed enhancements can effectively enhance the capability of \VI-aware \LCP.


\subsection{Overall performance of the two proposed presolve techniques}

\begin{table}[t]
	\renewcommand{\arraystretch}{1.4}
	\addtolength{\tabcolsep}{2pt}
	\centering
	\caption{Performance impact of using both \VI aggregation and \VI-aware \LCP.}
	\scriptsize
	\begin{tabular}{{|l|r|r|r|r|r|r|r|r|r|}} \hline
		&  & \multicolumn{3}{c|}{\Default} & \multicolumn{3}{c|}{\all} & \multicolumn{2}{c|}{Compare} \\ \hline
		Bracket &           \tblIns &             \tblS &               \tblT &               \tblN &             \tblS &               \tblT &               \tblN &               \tblT &               \tblN \\ \hline
		$\ge$    0   &      766       &       746      &       248.06   &       805.41   &       760      &       238.86   &       757.10   &         0.96   &         0.94   \\ \hline
		$\ge$   10   &      704       &       684      &       356.60   &      1228.32   &       698      &       341.19   &      1153.02   &         0.96   &         0.94   \\ \hline
		$\ge$  100   &      506       &       486      &       904.73   &      4444.69   &       500      &       833.40   &      4126.93   &         0.92   &         0.93   \\ \hline
		$\ge$ 1000   &      273       &       253      &      2341.07   &     14135.31   &       267      &      2004.70   &     12631.43   &         0.86   &         0.89   \\ \hline
		Affected &      428       &       408      &       356.69   &       439.63   &       422      &       325.32   &       392.27   &         0.91   &         0.89   \\ \hline
	\end{tabular}
	\label{table:4-MIPLIB2017}
\end{table}

Finally, we evaluate the overall performance impact of the two proposed presolve techniques---the \VI aggregation and the \VI-aware \LCP---in \HiGHS. 
To this end, we compare the setting \all, in which the two proposed presolve techniques are implemented, with the default setting of \HiGHS.
The comparison results are summarized in \cref{table:4-MIPLIB2017}.
From \cref{table:4-MIPLIB2017}, applying the two proposed presolve techniques leads to a much better performance.
Overall, \all achieves a reduction of 4\% in CPU time and 6\% in node number, with 14 more instances being solved.
For the 428 affected instances, a more pronounced reduction is observed: 9\% in CPU time and 11\% in node number.
Note that compared with applying the individual presolve technique (see \cref{table:vbca-MIPLIB2017} and \cref{table:3-MIPLIB2017}), applying both presolve techniques enables \rev{HiGHS} to solve more instances (760) and find reductions on more instances (428), making a better overall performance.

%% file: section_conclusion.tex
\section{Conclusion}\label{section-conclusion}

In this paper, we proposed two novel presolve techniques for \MIP problems: the \VI aggregation and \VI-aware \LCP.
The \VI aggregation attempts to aggregate multiple \VIs into a single inequality using the clique information,
which may not only tighten the LP relaxation but also reduce the size of the problem.
The proposed \rev{presolve} technique extends the standard clique merging \citep{Achterberg2020,Brito2021}, where reductions on \VIs involving non-binary variables are additionally performed.
The \VI-aware \LCP builds on standard \LCP  but involves the \VIs to derive more presolve reductions.
We showed that deriving the tightest lower or upper bound  for a variable from a system consisting of a linear constraint, variable bounds, integrality constraints, and the \VIs corresponding to the variable to be tightened can be conducted in linear time, which allows the proposed VI-aware LCP to be embedded into an \MIP solver.
We analyzed the relation of the proposed \VI-aware \LCP to the state-of-the-art approach of \citet{Achterberg2013a}, and demonstrated that the bounds derived by our proposed \rev{presolve} technique could be tighter.
Moreover, we also developed two \rev{enhancements} to further improve the effectiveness of the proposed \VI-aware \LCP.
By extensive computational experiments on the MIPLIB 2017 benchmark testset, we demonstrated the advantage of the proposed \VI-aware \LCP over the state-of-the-art approach of \citet{Achterberg2013a} and the effectiveness of the proposed \VI aggregation and \VI-aware \LCP in improving the performance of HiGHS.
In particular, using the two proposed presolve techniques, a reduction of 4\% in solving time and 6\% in node number on \HiGHS can be achieved.

%% file: section_appendix.tex
\newpage
    \author{Chen et al.}
    \title{Online appendix of the paper ``Exploiting Variable Implications in Presolve for Mixed Integer Programming''}

    {\noindent \Large \bf Online appendices of the paper ``Exploiting Variable Implications in Presolve for Mixed Integer Programming''}\\[20pt]

    \begin{center}
        {\large Wei-Kun Chen, Chang-Long Li, Zhao-Wei Wang, Yu-Hong Dai, Zi-Shuo Li, Meng Lu}
    \end{center}
    \vspace{0.5cm}

    \setcounter{page}{1}

    \begin{NoHyper}
        \renewcommand{\theHsection}{\thesection}

    \section{An illustrative example for function $w_r(\cdot)$}\label{appendixC}
      
    \begin{example}\label{example2} (\cref{example1} continued)
        To derive tighter bounds for the non-binary variable $x_5$, we consider the following three \VIs in \cref{example1}:
        \begin{subequations}
        	\begin{multicols}{3}
        		\noindent
        		\begin{align}
        			\label{ee-imp-1}
        			\tag{\ref*{ex-D}}
        			&{x_5 > 0.5 \to x_2 = 0},
        		\end{align}
        		\begin{align}
        			\label{ee-imp-2}
        			\tag{\ref*{e-imp4}}
        			&x_5 < 1 \to x_2 = 1,
        		\end{align}
        		\begin{align}
        			\label{ee-imp-3}
        			\tag{\ref*{e-imp5}}
        			&x_5 > 2 \to x_3 = 1.
        		\end{align} 
        	\end{multicols}
        \end{subequations}\noindent
        Here, \VIs in \eqref{e-imp3}--\eqref{e-x4} are not considered, as we have already fixed $x_1 = 1$ in \cref{example1}.
        From \VIs in \eqref{ee-imp-2}--\eqref{ee-imp-3}, we can obtain $\mathcal{B}_1^<(5) = \{2\}$, $\ell_{25} = 1$, and $\mathcal{B}_1^>(5) = \{3\}$, $u_{35} = 2$,
        {from which we can decompose the domain $[0, 3]$ into $[0,1] \cup [1,2] \cup [2,3]$.
        From \VIs \rev{~in} \eqref{ee-imp-1}--\eqref{ee-imp-2}, we obtain $\D = \left(0.5, 1\right)$.
        }
        Therefore, we obtain
        \begin{equation*}
            [0, 0.5] \cup [1, 2] \cup [2, 3].
        \end{equation*}
        On each of the three corresponding open intervals $(0, 0.5),~ (1,2),~ (2,3)$, function \rev{$w_5(\cdot)$} is linear.
        Specifically, we obtain
        \begin{equation}\label{ex-w}
            \begin{aligned}
                w_{5}(d) & = \sum_{\substack{i \in \{2\},~ d < \ell_{i5},~ d \notin \D}} a_i  + \sum_{\substack{i \in \{3\},~ d > u_{i5},~ d \notin \D}} a_i  + 0.5d \\
                & =  \mathbb{I}{\{d \le 0.5\}} + 0.9 \mathbb{I}{\{d>2\}} + 0.5d
                = \left \{  \begin{array}{ll}
                    0.5d + 1 & ~\text{if}~d \in [0,0.5], \\[5pt]
                    0.5d &  ~\text{if}~d \in [1,2],\\[5pt]
                    0.5d + 0.9& ~\text{if}~d \in (2,3],
                    \end{array}
                \right.
            \end{aligned}
        \end{equation}
        where $\mathbb{I}\{\cdot\}$ is the 0-1 indicator function. \cref{fig:w3} plots the function value of \rev{$w_5(\cdot)$}.
        As shown in \cref{fig:w3}, $w_5(\cdot)$ is a piecewise linear function; however, it is not monotonic on the {feasible domain $[0, 0.5] \cup [1, 3]$.}
        
        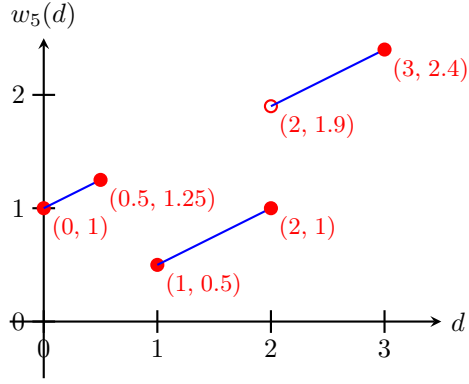
\begin{figure}[htbp]
            \centering
            \begin{tikzpicture}[
                scale=1.5,
                thick,
                >=stealth
                ]
                
                \draw[->] (-0.3,0) -- (3.5,0) node[right] {$d$};
                \draw[->] (0,-0.5) -- (0,2.5) node[above] {$w_5(d)$};
                
                \foreach \x in {0,1,2,3}
                \draw (\x,0.1) -- (\x,-0.1) node[below=-1pt] {\(\x\)};
                
                \foreach \y in {0,1,2}
                \draw (0.1,\y) -- (-0.1,\y) node[left=-1pt] {\(\y\)};
                
                \filldraw[red] (0,1) circle (1.5pt);
                \node[below right=-1pt] at (0,1) {\textcolor{red}{\small (0,\;1)}};
                \draw[blue, thick] (0,1) -- (0.5,1.25);
                \filldraw[red] (0.5,1.25) circle (1.5pt);

                \node[below right=-1pt] at (0.5,1.25) {\textcolor{red}{\small (0.5,\;1.25)}};
                
                \filldraw[red] (1,0.5) circle (1.5pt);
                \node[below right=-1pt] at (1,0.5) {\textcolor{red}{\small (1,\;0.5)}};
                \draw[blue, thick] (1,0.5) -- (2,1.0);
                
                \filldraw[red] (2,1.0) circle (1.5pt);
                \node[below right=-1pt] at (2,1.0) {\textcolor{red}{\small (2,\;1)}};
                
                \draw[red, fill=white] (2,1.9) circle (1.5pt);
                \node[below right=-1pt] at (2,1.9) {\textcolor{red}{\small (2,\;1.9)}};
                \draw[blue, thick] (2,1.9) -- (3,2.4);
                
                \filldraw[red] (3,2.4) circle (1.5pt);
                \node[below right=-1pt] at (3,2.4) {\textcolor{red}{\small (3,\;2.4)}};
                
            \end{tikzpicture}
            \caption{\centering Illustration of function $w_5(\cdot)$ in \cref{example2}.}
            \label{fig:w3}
        \end{figure}
    \end{example}
    
    \begin{example}(\cref{example1} continued)\label{example3}
    	Consider the \rev{function $w_5(\cdot)$} computed in \eqref{ex-w}.
    	First, we note that all the endpoints are {$\dunderbar_1 = 0,~ \bar{d}_1 = 0.5,~ \dunderbar_2 = 1,~ \bar{d}_2 = \dunderbar_3 = 2,~ \bar{d}_3 = 3$.}
    	Applying \cref{non-01-prop}, we obtain
    	\begin{equation*}
    		\LB_1 = 0,~ \UB_1 = 0.5,~ \LB_2 = 1,~ \UB_2 = 2,~ \LB_3 = 2,~ \UB_3 = 2.2.
    	\end{equation*}
    	Therefore, we can tighten the upper bound of $x_5$ as {$\max\{\UB_1, \UB_2, \UB_3\} = 2.2 < 3$}. 
    \end{example}


       

    \section{Comparing the proposed \VI-aware \LCP with the presolve technique in \citet{Achterberg2013a}: examples}\label{appendixA}

        \begin{example}\label{tmpex}
            Consider the following knapsack constraint 
            \begin{subequations}
                \begin{align}
                    & x_1 + x_2 + x_3 + x_4 + 0.1x_5 \le 2, \label{example-linear}\\
                    & x_i \in \left\{0,1\right\},~ i = 1,2,3,4,5, \label{example-bounds}
                \end{align}
            \end{subequations}
            and the following \VIs of $x_6 \in [0,4]$:
            \begin{subequations}
                \begin{multicols}{2}
                    \noindent
                    \begin{align}
                        &x_6 < 3 \to x_1 = 1, \label{example-imp1} \\
                        &x_6 < 3 \to x_2 = 1, \label{example-imp2} \\
                        &x_6 < 2 \to x_3 = 1, \label{example-imp3}
                    \end{align}
                    \begin{align}
                        &x_6 < 2 \to x_4 = 1, \label{example-imp4} \\
                        &x_6 > 1 \to x_5 = 1. \label{example-imp5}
                    \end{align}
                \end{multicols}
            \end{subequations}\noindent
            For \VIs in \eqref{example-imp1}--\eqref{example-imp5}, we obtain $\B_1^<(6) = \left\{1,2,3,4\right\},~ \ell_{16} = \ell_{26} = 3,~ \ell_{36} = \ell_{46} = 2$,  $\B_1^>(6) = \{5\}$, and $u_{56} = 1$.
            Using the presolve technique in \citet{Achterberg2013a} to tighten the lower bound of variable $x_6$, we obtain a new lower bound for $x_6$:
            \begin{align}
                \LBhat = \max \left\{\ell_{i_0 6},~ i_0 \in \{1,2,3,4\} ~\Bigm|~ \sum_{\ell_{i_0 6} \le \ell_{i 6},~ i \in \{1,2,3,4\}} a_i > 2\right\} = 2.
            \end{align}
            However, using the proposed \VI-aware \LCP, we can derive a tighter lower bound for $x_6$:
            \begin{align}
                \LB = \min \left\{d \in \left[0, 4\right] ~\Bigm|~ \sum_{\substack{i \in \B^{<}_1(6),~ d < \ell_{i6}}} a_i  + \sum_{\substack{i \in \B^{>}_1(6),~ d > u_{i6}}} a_i \le 2\right\} = 3. 
            \end{align}
            This shows that  the proposed \VI-aware \LCP can indeed derive strictly tighter bounds than those derived by the presolve technique in \citet{Achterberg2013a}.
        \end{example}

        \begin{example}\label{tmp-eex}
            Consider a variant of \cref{tmpex}, where the knapsack constraint \eqref{example-linear} is replaced by the following general linear constraint:
            \begin{align}\label{tmpcons}
                & x_1 + x_2 + x_3 + x_4 + 0.1 x_5 + 0.2 x_6 \le 2.2,
            \end{align}
            and the \VIs of $x_6 \in [0,4]$ are  \eqref{example-imp1}--\eqref{example-imp4}.
            Using  the proposed \VI-aware \LCP, we can 
            tighten the lower bound of $x_6$ as 
            \begin{align}
                \LB = \min \left\{d \in \left[0, 4\right] ~\Bigm|~ \sum_{\substack{i \in \B^{<}_1(6),~ d < \ell_{i6}}} a_i   + 0.2 d \le 2.2\right\} = 3.
            \end{align}
                However, using the state-of-the-art approach of \citet{Achterberg2013a}, we need to first relax the non-binary variable $x_6$ from constraint \eqref{tmpcons} using, e.g., the variable bound $x_6 \ge 0$, obtaining a knapsack constraint: 
            \begin{align}
                \label{relaxed-knap-tmp}
                & x_1 + x_2 + x_3 + x_4 + 0.1x_5 \le 2.2;
            \end{align}
            \rev{then} the lower bound of $x_6$ derived from \eqref{example-bounds}, \eqref{example-imp1}--\eqref{example-imp4}, and \eqref{relaxed-knap-tmp} is $\LB'=2< \LB$.
            This demonstrates that the proposed \VI-aware \LCP (without relaxing non-binary variables from the constraint) could derive strictly tighter variable bounds than those derived by the state-of-the-art approach of \cite{Achterberg2013a}.
        \end{example}

        \section{Algorithm framework for the proposed \VI-aware \LCP}\label{appendixB}

        \subsection{Algorithm description}\label{appendixB1}
        The algorithm for applying the \VI-aware \LCP is summarized in \cref{alg:boundtightening-binary}.

        \begin{algorithm}[!htbp]
            \caption{Implementation of the \VI-aware \LCP}
            \label{alg:boundtightening-binary}
            {\fontsize{8.5pt}{12pt}\selectfont
                \begin{algorithmic}[1]
                    \State \textbf{Input:} Normalized constraint \eqref{linear}, variable upper bounds $0 \leq x \leq u$, integer set $\I$, clique table $\mathfrak{C}$, and \VI graph $G$.
                    \State \textbf{Output:} Tightened variable bounds for all variables.
                    \commentInline{Tighten the lower and upper bounds of binary variables}
                    \State Initialize $w_r(0) := 0$ and $w_r(1) := a_r$ for all $r \in \B$, and mark that $a_r$ is included in the computation of $w_r(1)$. \label{line-bin-init}
                    \For{$j \in \N \backslash \B$ with $a_j > 0$}\label{line-nonbin-loop-start}
                        \For{\VIs $x_r = 0 ~\to~ x_j \ge \ell_{rj}$ in $\G$} \label{line-nonbin-imp-1}
                            \State Set $w_r(0) := w_r(0) + a_j \ell_{rj}$.
                        \EndFor
                        \For{\VIs $x_r = 1 ~\to~ x_j \ge \ell_{rj}$ in $\G$}\label{line-nonbin-imp-2}
                            \State Set $w_r(1) := w_r(1) + a_j \ell_{rj}$.
                        \EndFor
                    \EndFor\label{line-nonbin-loop-end}

                    \For{$ i \in \B$ with $a_i > 0$} \label{line-bin-loop-start}
                        \For{$\C = \left(\C^+, \C^-\right) \in \mathfrak{C}$ with $i \in \C^-$}	\label{line-bin-visit-clique-start}
                            \For{$r \in \C^+$ and $a_i$ is not included in the computation of $w_r(1)$}
                                \State Set $w_r(1) := w_r(1) + a_i$, and mark that $a_i$ is included in the computation of $w_r(1)$. \label{line-bin-w1}
                            \EndFor
                            \For{$r \in \C^- \backslash \{i\}$ and $a_i$ is not included in the computation of $w_r(0)$}
                                \State Set $w_r(0) := w_r(0) + a_i$, and mark that $a_i$ is included in the computation of $w_r(0)$.\label{line-bin-w0}
                            \EndFor
                        \EndFor \label{line-bin-visit-clique-end}
                    \EndFor \label{line-bin-loop-end}

                    \For{$r \in \B$}\label{line-bin-fix-start}
                        \State If {$w_r(0) > b_0$}, then {fix $x_r = 1$}.
                        \State If {$w_r(1) > b_0$}, then {fix $x_r = 0$}.
                    \EndFor\label{line-bin-fix-end}
                    

                    \commentInline{Tighten the lower and upper bounds of non-binary variables}
                    \For{$r \in \N \backslash \B$} \label{line-nbin-init}
                        \State Set $w_r(0) := 0$ and $w_r(u_r) := a_r u_r$. \label{line-bin-imp-1-start}
                        \For{{$i \in \B_1^<(r)$ with $a_i > 0$}} 
                            \State Set $w_r(0) := w_r(0) + a_i$.
                        \EndFor
                        \For{{$i \in \B_1^>(r)$ with $a_i > 0$}}
                            \State Set $w_r(u_r) := w_r(u_r) + a_i$.
                        \EndFor \label{line-bin-imp-2-end}
                    
                        \If{$w_r(0) > b_0$ or $w_r(u_r) > b_0$} \label{line-nonb-lbub-start}
                            \State {Compute 
                            	$\left\{\dunderbar_k, \bar{d}_k\right\}_{k=1}^{\tau}$ so that \eqref{d-define-1} and \eqref{d-define-2} hold.}
                            \State For each $k \in \{1, \ldots, \tau\}$, use \cref{non-01-prop} to calculate $\LB_k$ and $\UB_k$.
                            \State Set $\LB$ and $\UB$ as in \eqref{minlbk}.
                        \EndIf \label{line-nonb-lbub-end}
                    \EndFor \label{line-nbin-end}
                \end{algorithmic}
            }
        \end{algorithm}

        \subsection{Techniques for improving the performance of \cref{alg:boundtightening-binary}}\label{appendixB2}

        We present two techniques to improve the performance of \cref{alg:boundtightening-binary} as follows.

        \textit{Early termination.} 
            To further avoid unnecessary computations of checking the clique table, we can track the temporary implied activities $w_r(0),~w_r(1)$, and the maximal remaining sum of coefficients \rev{$\bar{w} := \sum_{i \in \hat{\B}} a_i$} in the loop in steps \ref{line-bin-loop-start}--\ref{line-bin-loop-end}, where $\hat{\B} \subseteq \B$ denotes the set of binary variables $x_i$ with $a_i > 0$ that has not been processed in step \ref{line-bin-loop-start}.
            If $w_r(0) + \bar{w} \le b_0$ or $w_r(1) + \bar{w} \le b_0$ holds in the current iteration, then the final implied activities $w_r(0)$ or $w_r(1)$ cannot exceed $b_0$, implying that variable $x_r$ cannot be fixed to one or zero.
            {We terminate steps \ref{line-bin-loop-start}--\ref{line-bin-loop-end} as soon as every binary variable is either fixed (via $w_r(0) > b_0$ or $w_r(1) > b_0$), or proven to be unfixable (via $w_r(0) + \bar{w} \le b_0$ and $w_r(1) + \bar{w} \le b_0$).}
            Note that for $r \in \hat{\B}$, we can provide an improved condition to ensure that $x_r$ cannot be fixed to one.
            Specifically, for $r \in \hat{\B}$, fixing $x_r = 0$ cannot deduce $x_r = 1$ (from \cref{assum1}(iii)); therefore, if $w_r(0) + \bar{w} - a_r \le b_0$ holds in the current step, then the final implied activity $w_r(0)$ cannot exceed $b_0$, implying that  $x_r$ cannot be fixed to one.
        
        \textit{Sorting variables.} 
            In order to identify uninteresting binary variables as soon as possible, the order in which we iterate binary variables with positive coefficients in step \ref{line-bin-loop-start} of \cref{alg:boundtightening-binary} becomes important.
            Here, we follow \cite{Achterberg2013a} to traverse binary variables with positive coefficients in a non-decreasing order of 
            $$\left|\{\left(\C^+, \C^-\right) \in \mathfrak{C} ~\mid~ \bar{x}_i \in \C^-\}\right| - \frac{10 a_i}{b_0}.$$ 
            With this order, variables with larger coefficients and fewer \VIs are likely to be traversed first, thereby forcing $w_r(0) + \bar{w} \le b_0$ or $w_r(1) + \bar{w} \le b_0$ to be detected quickly (if it holds).

    \end{NoHyper}